\documentclass[a4paper,12pt]{article}%
\usepackage[T1]{fontenc}
\usepackage[english]{babel}
\usepackage[ansinew]{inputenc}
\usepackage{lmodern}
\usepackage{amsmath,amssymb,amsfonts,amstext}
\usepackage{amsthm}
\usepackage{graphicx}
\usepackage{amsmath}
\usepackage{amsfonts}
\usepackage{amssymb}
\usepackage{float}
\theoremstyle{definition}
\newtheorem{theorem}{Theorem}[section]
\newtheorem{lemma}{Lemma}[section]
\newtheorem{remark}{Remark}[section]

\numberwithin{equation}{section}
\begin{document}

\title{A new improved explicit estimate\\for $\zeta\left(  1/2+it\right)  $}
\author{Michael Revers}
\date{}
\maketitle

\begin{abstract}
In this paper, we present an improved explicit subconvexity result for the Riemann zeta function $\zeta\left( s\right)$ along the critical line $s=1/2+it$, given by Hiary, Patel and Yang in 2024. This new bound is derived by combining a refined, explicit version of the van der Corput method together with computational calculations.



\end{abstract}

\section{Introduction}

\noindent A well known problem in analytic number theory is to bound the
growth rate of the Riemann zeta function along vertical lines $s$ inside the
critical strip. Especially explicit bounds for zeta along the critical line
$s=1/2+it$ have broad implications for the distribution of the non trivial
zeros inside the critical strip. They are used to derive zero-free regions
\cite{Ford}, \cite{Mossinghoff}, \cite{Yang}, zero-density estimates
\cite{Kadiri} and also bounds on the argument of the zeta function on the
critical line \cite{Hasanalizade}, \cite{Trudgian}. Moreover, the problem has
led to deep ideas in estimating exponential sums, in particular the method of
exponential pairs, see \cite{GrahamKolesnik}, (\cite{Titchmarch}, page 116).

\medskip\noindent Appealing to the Riemann-Siegel formula, the earliest
explicit convexity bound result is due to Lehman (\cite{Lehman}, Lemma 2) who
showed that
\[
\left\vert \zeta\left(  1/2 +it\right)  \right\vert \leq\frac{4}{\left(
2\pi\right)  ^{1/4}}t^{1/4},\quad t\geq32/\pi.
\]
The Weyl-Hardy-Littlewood method, see (\cite{Titchmarch}, section 5.3) detects
a certain amount of cancellation inside the main Riemann-Siegel sum and
combined together with the van der Corput method \cite{Corput1},
\cite{Corput2} the size of the bound can be improved to $t^{1/6}\log t$.
Explicit estimates of this type that means estimates of the sort%
\[
\left\vert \zeta\left(  1/2 +it\right)  \right\vert \leq C_{1}t^{1/6}\log
t,\quad t\geq C_{2}%
\]
with effective computable values $C_{1}$ and $C_{2}$ are generally hard to
derive and moreover much harder to improve. For numerical purposes it is
necessary to have such bounds available since they offer an unconditional
measure against which one can compare large values of $\left\vert \zeta\left(
1/2+it\right)  \right\vert $ found by certain numerical algorithms. The first
available result of which we are aware is due to Cheng and Graham
\cite{ChengGraham} when they proved in 2004
\[
\left\vert \zeta\left(  1/2 +it\right)  \right\vert \leq3t^{1/6}\log t,\quad
t\geq e.
\]
This result is further improved by Platt and Trudgian \cite{PlattTrudgian} in
2015 to%
\[
\left\vert \zeta\left(  1/2 +it\right)  \right\vert \leq0.732t^{1/6}\log
t,\quad t\geq2.
\]
The next step is due to Hiary \cite{Hiary1} in 2016 when he showed that%
\[
\left\vert \zeta\left(  1/2 +it\right)  \right\vert \leq0.63t^{1/6}\log
t,\quad t\geq3.
\]
Eight years later, in 2024, Hiary, Patel and Yang \cite{Hiary2} improved this
further to%
\[
\left\vert \zeta\left(  1/2 +it\right)  \right\vert \leq0.618t^{1/6}\log
t,\quad t\geq3.
\]
Numerical computations show that $\max_{t\geq3}\left\vert \zeta\left(
1/2+it\right)  \right\vert /t^{1/6}\log t>0.5075$. This maximum is attained at
least near the values $t=45.62$, $t=63.06$ and $t=108.99$. So the leading
constant in the above mentioned four estimates cannot be reduced below
$0.5075$ without an additional assumption on the size of $t$. The exponent
$1/6$ in the above mentioned estimates is also hard to improve. One of the
sharpest result so far is due to Huxley \cite{Huxley} form 2005 who proved
that%
\[
\zeta\left(  1/2 +it\right)  \ll t^{32/205}\log^{\gamma}t
\]
for some constant $\gamma$. Note $32/205=0.15609\ldots$. The actual sharpest
bound is due Bourgain \cite{Bourgain} from 2016, when he showed that for every
$\varepsilon>0$ we have
\[
\zeta\left(  1/2 +it\right)  \ll t^{13/84+\varepsilon}.
\]
Note $13/84=0.1547\ldots$. According to the Lindel\"of hypothesis
\cite{Lindelof}, the growth rate is expected to be slower than any fixed power
of $t$ that means%
\[
\zeta\left(  1/2 +it\right)  \ll t^{\varepsilon}%
\]
for every $\varepsilon>0$. Thus, all recorded subconvexity bounds in the
literature appear to be far away from the truth, at least for sufficiently
large values of $t$.

\medskip\noindent Now, let us turn again to the results on the explicit
subconvexity bounds obtained by Hiary \cite{Hiary1} and Platt-Trudgian
\cite{PlattTrudgian}. The accompanying stories around these results take a
dramatic turn. As has recently been pointed out by K. Ford and also discussed
in a preprint by J. Arias de Reyna, see also in more detail Patel
(\cite{Patel}, section 2), both results are affected by an incorrect
computational constant originating from a fatal flaw in the proof of the so
called Cheng-Graham lemma. In (\cite{ChengGraham}, Lemma 2) a certain
exponential sum is estimated by the value $1/\left(  \pi\theta\right)  +1$
though the constant $1/\pi$ in this bound should be $2/\pi$. Now, following
the computations in \cite{Hiary1} using the constant $2/\pi$ in place of the
erroneous $1/\pi$ then one arrives at the constant 0.77, which differs greatly
from the constant $0.63$ in the Hiary result from 2016. It becomes
unmistakable that the true value of the newer result by Hiary, Patel and Yang
\cite{Hiary2} from 2024 must be seen not only in the much better constant
0.618 but more important that it leaves intact all works in the literature
that depend directly or indirectly on \cite{Hiary1} or on \cite{PlattTrudgian}.

\medskip\noindent This paper has two objectives. Firstly, we aim to improve
the current best explicit bound further. Secondly, we aim to supplement the
original proof with important details that we believe are missing from the
paper \cite{Hiary2}. In addition, to estimate inside the intermediate region,
where t is too large for the Riemann-Siegel bound but too small for the van
der Corput bound, extensive computational calculations are required.
These calculations are performed in subsection \ref{ComputationalBottleneck}.

\medskip\noindent Generally to say, proofs on explicit bounds remain much more
fragile compared to other mathematical results, because the main outcome is a
single number. This is ultimately why we present an extensive and detailed
proof, going into more detail than is strictly necessary, just to demonstrate evidence of
correctness and ensure nothing important is missed.

\begin{theorem}
\label{theorem1}\mbox{}

\noindent(a) For $t\geq3$ we have
\[
\left\vert \zeta\left(  1/2+it\right)  \right\vert \leq0.611t^{1/6}\log t.
\]
(b) For $t\geq8.97\cdot10^{17}$ we have
\[
\left\vert \zeta\left(  1/2+it\right)  \right\vert \leq0.566t^{1/6}\log t.
\]

\end{theorem}

\subsection{Notation}

\noindent Throughout this paper we denote by $e\left(  x\right)  = \exp\left(
2\pi ix\right)  $. Further $\left\lfloor .\right\rfloor $ and $\left\lceil
.\right\rceil $ have their usual meanings that is $\left\lfloor x\right\rfloor
$ denotes the floor function with $\left\lfloor x\right\rfloor =\max\left\{
m\in\mathbb{Z}:m\leq x\right\}  $ and $\left\lceil x\right\rceil $ denotes the
ceiling function with $\left\lceil x\right\rceil =\min\left\{  n\in
\mathbb{Z}:n\geq x\right\}  .$

\section{Preliminaries}

\label{Preliminaries}

\begin{lemma}
[Computational Bound]\label{ComputationalBound} For $3\leq t\leq200$ we have%
\[
\left\vert \zeta\left(  1/2 +it\right)  \right\vert \leq 0.611 t^{1/6}\log t.
\]

\end{lemma}

\begin{proof}
The proof relies on an implemented Euler-MacLaurin formula using interval
arithmetic and is carried out in (\cite{Hiary1}, Lemma 2.2). Note, that here
one can take the better constant $0.595$, see (\cite{Hiary1}, page 528, formula 15).
\end{proof}

\begin{lemma}
[Riemann-Siegel Bound]\label{RiemannSiegel}\mbox{}

\noindent(a) For $t\geq200$ we have
\[
\left\vert \zeta\left(  1/2 +it\right)  \right\vert \leq\frac{4t^{1/4}%
}{\left(  2\pi\right)  ^{1/4}}-2.0895.
\]
(b) For $200\leq t\leq3.3\cdot10^{7}$ we have
\begin{align*}
\left\vert \zeta\left(  1/2 +it\right)  \right\vert  &  \leq\frac{4t^{1/4}%
}{\left(  2\pi\right)  ^{1/4}}-2.0895\\
&  \leq 0.611t^{1/6}\log t.
\end{align*}
(c) For $3.3\cdot10^{7}\leq t\leq3.8\cdot10^{7}$ we have
\begin{align*}
\left\vert \zeta\left(  1/2 +it\right)  \right\vert  &  \leq\frac{4t^{1/4}%
}{\left(  2\pi\right)  ^{1/4}}-2.8805\\
&  \leq 0.611t^{1/6}\log t.
\end{align*}
\end{lemma}

\begin{proof}
\noindent(a) We start with the Riemann-Siegel formula and apply the triangle
inequality to the main sum and to the Gabke remainder terms \cite{Gabke},
following the procedure as in \cite{Hiary1}. For $t\geq200$ and $n_{1}%
=\left\lfloor \sqrt{t/2\pi}\right\rfloor \geq5$ we have%
\begin{align*}
&  \left\vert \zeta\left(  1/2+it\right)  \right\vert \\
&  \leq2\left\vert \sum_{n=1}^{n_{1}}n^{-1/2+it}\right\vert +\cos\frac{\pi}%
{8}\cdot\left(  2\pi\right)  ^{1/4}t^{-1/4}+0.127t^{-3/4}\\
&  \leq2\sum_{n=1}^{5}\frac{1}{\sqrt{n}}+2\int_{5}^{n_{1}}\frac{1}{\sqrt{x}%
}dx+\cos\frac{\pi}{8}\cdot\left(  2\pi\right)  ^{1/4}t^{-1/4}+0.127t^{-3/4}\\
&  =2\sum_{n=1}^{5}\frac{1}{\sqrt{n}}+4\sqrt{n_{1}}-4\sqrt{5}+\cos\frac{\pi
}{8}\cdot\left(  2\pi\right)  ^{1/4}t^{-1/4}+0.127t^{-3/4}\\
&  \leq\frac{4t^{1/4}}{\left(  2\pi\right)  ^{1/4}}+2\sum_{n=1}^{5}\frac
{1}{\sqrt{n}}-4\sqrt{5}+\cos\frac{\pi}{8}\cdot\left(  2\pi\right)  ^{1/4}%
\cdot200^{-1/4}+0.127\cdot200^{-3/4}\\
&  =\frac{4t^{1/4}}{\left(  2\pi\right)  ^{1/4}}-2.08958\ldots\leq
\frac{4t^{1/4}}{\left(  2\pi\right)  ^{1/4}}-2.0895.
\end{align*}
(b) For $200\leq t\leq3.3\cdot10^{7}$, applying some standard calculations
together with (a), reveal that%
\[
\frac{4t^{1/4}}{\left(  2\pi\right)  ^{1/4}}-2.0895\leq 0.611 t^{1/6}\log t.
\]
(c) Let now $3.3\cdot10^{7}\leq t\leq3.8\cdot10^{7}$ and $t_{0}=3.3\cdot
10^{7}$. First we observe that%
\[
n_{1}=\left\lfloor \sqrt{\frac{t}{2\pi}}\right\rfloor \geq\left\lfloor
\sqrt{\frac{t_{0}}{2\pi}}\right\rfloor =\left\lfloor 2291.75\ldots\right\rfloor
=2291.
\]
Following the procedure in a similar way as in (a), we obtain%
\begin{align*}
&  \left\vert \zeta\left(  1/2+it\right)  \right\vert \\
&  \leq2\left\vert \sum_{n=1}^{n_{1}}n^{-1/2+it}\right\vert +\cos\frac{\pi}%
{8}\cdot\left(  2\pi\right)  ^{1/4}t^{-1/4}+0.127t^{-3/4}\\
&  \leq2\sum_{n=1}^{2291}\frac{1}{\sqrt{n}}+2\int_{2292-1/2}^{n_{1}+1/2}%
\frac{dx}{\sqrt{x}}+\cos\frac{\pi}{8}\cdot\left(  2\pi\right)  ^{1/4}%
t_{0}^{-1/4}+0.127t_{0}^{-3/4}\\
&  =2\sum_{n=1}^{2291}\frac{1}{\sqrt{n}}+2\int_{2292-1/2}^{n_{1}}\frac
{dx}{\sqrt{x}}+2\int_{n_{1}}^{n_{1}+1/2}\frac{dx}{\sqrt{x}}\\
&  +\cos\frac{\pi}{8}\cdot\left(  2\pi\right)  ^{1/4}t_{0}^{-1/4}%
+0.127t_{0}^{-3/4}\\
&  \leq\frac{4t^{1/4}}{\left(  2\pi\right)  ^{1/4}}+2\sum_{n=1}^{2291}\frac
{1}{\sqrt{n}}-4\sqrt{2291.5}+\frac{1}{\sqrt{2291}}\\
&  +\cos\frac{\pi}{8}\cdot\left(  2\pi\right)  ^{1/4}t_{0}^{-1/4}%
+0.127t_{0}^{-3/4}\\
&  \leq\frac{4t^{1/4}}{\left(  2\pi\right)  ^{1/4}}-2.88052\ldots\leq\frac{4t^{1/4}%
}{\left(  2\pi\right)  ^{1/4}}-2.8805.
\end{align*}
By some standard arguments, we finally obtain for our range for $t$%
\[
\frac{4t^{1/4}}{\left(  2\pi\right)  ^{1/4}}-2.8805\leq 0.611 t^{1/6}\log t.
\]
\end{proof}

\begin{lemma}
[Refined Kusmin-Landau Bound]\label{KusminLandau}Let $f:\left[  a,b\right]
\rightarrow\mathbb{R}$ with a monotonic and continuous derivative on $\left[
a,b\right]  $ satisfying
\[
z+U\leq f^{\prime}\left(  x\right)  \leq z+V.\quad x\in\left[  a,b\right]
\]
for some integer $z$ and $0<U\leq V<1.$ Then
\[
\left\vert \sum_{a\leq n\leq b}e\left(  f\left(  n\right)  \right)
\right\vert \leq\frac{1}{\pi}\left(  \frac{1}{U}+\frac{1}{1-V}\right)  .
\]

\end{lemma}

\begin{proof}
\noindent This is Lemma 2.3 in \cite{Hiary2} with a slight modified version, here mentioned as the generalized
Cheng-Graham lemma.
\end{proof}

\begin{remark}
\label{KusminRemark} From the original proofs for Lemma 2.3 in \cite{Hiary2}
it follows that we can directly change the summation to closed intervals
$\left[  a,b \right]  $ as well as to open intervals $\left(  a,b \right)  $
instead of $\left[  a,b \right)  $ since the bound in this Lemma is
independent of the length of the summation. See also (\cite{Hiary2}, page 213).
\end{remark}

\begin{remark}
We also mention that the proof by Cheng-Graham \cite{ChengGraham} follows the
original proofs from Kusmin \cite{Kusmin} and Landau \cite{Landau}. The more
or less trivial generalization can be seen in the supplement of the integer
$z$ and in the extension to asymmetric boundary values $U$ and $V$. For this
reason we renamed the Lemma following the historic line.
\end{remark}

\begin{lemma}
[Weyl differentiation]\label{Weyl}For $f:\left[  N+1,N+L\right]
\rightarrow\mathbb{R}$ and positive integers $L$ and $M$ we have%
\[
\left\vert \sum_{n=N+1}^{N+L}e\left(  f\left(  n\right)  \right)  \right\vert
^{2}\leq\frac{L+M-1}{M}\left(  L+2\sum_{m=1}^{M}\left(  1-\frac{m}{M}\right)
\left\vert s_{m}^{\prime}\left(  L\right)  \right\vert \right)  ,
\]
where
\[
s_{m}^{\prime}\left(  L\right)  =\sum_{r=N+1}^{N+L-m}e\left(  f\left(
m+r\right)  -f\left(  r\right)  \right)  .
\]
If $m \geq L$, then $s_{m}^{\prime}\left(  L\right)  $ is an empty sum with
value $0$.
\end{lemma}

\begin{proof}
\noindent This is Lemma 2.4 in \cite{Hiary2}.
\end{proof}

\begin{lemma}
\label{Functionestimate}Let $m,r,L,K\geq2$ be positive integers and $t,K_{0}$
be positive numbers with $K \geq K_{0}>1$. Further we define
\begin{align*}
f\left(  x\right)   &  =\frac{t}{2\pi}\log\left(  Kr+x\right)  ,\\
g\left(  x\right)   &  =f\left(  m+x\right)  -f\left(  x\right)  ,\\
W  &  =\frac{\pi\left(  r+1\right)  ^{3}K^{3}}{t},\\
\lambda &  =\frac{\left(  r+1\right)  ^{3}}{r^{3}},\\
\mu &  =\frac{1}{2}\lambda^{2/3}\left(  1+\frac{1}{\left(  1-K_{0}%
^{-1}\right)  \lambda^{1/3}}\right)  .
\end{align*}
Then, for $m<L\leq K$, we have%
\begin{align*}
&  g^{\prime}\left(  L-1-m\right)  -g^{\prime}\left(  0\right) \\
&  \leq g^{\prime}\left(  K-1-m\right)  -g^{\prime}\left(  0\right) \\
&  \leq\frac{mK}{W}\mu.
\end{align*}
\end{lemma}

\begin{proof}
\noindent For the first inequality we use that $g^{\prime}$ is monotonically
increasing. The proof for the second inequality follows the steps in
(\cite{Hiary2}, formulas 5.1 and 5.2) and can be arranged by some routine
standard arguments.
\end{proof}

\begin{lemma}
[Improved second derivative test, Part 1]\label{SecondderivativePart1} Let
$m,r,L,K\geq2$ be positive integers and $t,K_{0}$ be positive numbers with $K
\geq K_{0}>1$. Further let $0<\Delta<1/2$ and $f,g,W,\lambda,\mu$ to be
defined as in Lemma \ref{Functionestimate}. Then we have

\noindent(a) For $m<K$ we have
\[
\frac{m}{W}\leq g^{\prime\prime}\left(  x\right)  \leq\frac{m}{W}\lambda
,\quad0\leq x\leq K-m.
\]
(b) For $m<L\leq K$ we have
\begin{align*}
&  \left\vert \sum_{n=0}^{L-1-m}e\left(  g\left(  n\right)  \right)
\right\vert \\
&  \leq\left(  \frac{mK}{W}\mu-1\right)  \left(  1+\frac{2W\Delta}{m}+\frac
{2}{\pi\Delta}\right)  +P\left(  \varepsilon_{1}\right)  +P\left(
1-\varepsilon_{2}\right) -\frac{2}{\pi\Delta},
\end{align*}
with%
\[
P\left(  \varepsilon\right)  =\left(  1+\frac{2W\Delta}{m}+\frac{2}{\pi\Delta
}\right)  \varepsilon+p\left(  \varepsilon\right)  ,
\]%
\[
p\left(  \varepsilon\right)  =\left\{
\begin{array}
[c]{ll}%
1+\frac{2}{\pi\Delta}+\frac{W}{m}\left(  \Delta-\varepsilon\right)  &
\text{for }\varepsilon\in\left[  0,\Delta\right)  ,\\
\frac{1}{\pi}\left(  \frac{1}{\Delta}+\frac{1}{\varepsilon}\right)  &
\text{for }\varepsilon\in\left[  \Delta,1-\Delta\right]  ,\\
0 & \text{for }\varepsilon\in\left(  1-\Delta,1\right]  ,
\end{array}
\right.
\]
and%
\begin{align*}
\varepsilon_{1}  &  =g^{\prime}\left(  0\right)  -\left\lfloor g^{\prime
}\left(  0\right)  \right\rfloor ,\\
\varepsilon_{2}  &  =g^{\prime}\left(  L-1-m\right)  -\left\lfloor g^{\prime
}\left(  L-1-m\right)  \right\rfloor .
\end{align*}
\end{lemma}

\begin{lemma}
[Improved second derivative test, Part 2]\label{SecondderivativePart2} Let
$t\geq5\cdot10^{6}$, $\phi\in\left[  1/3, 0.35\right]$, $K=\left\lceil
t^{\phi}\right\rceil $, $n_{1}=\left\lfloor \sqrt{t/2\pi}\right\rfloor $,
$R=\left\lfloor n_{1}/K \right\rfloor $ and $4\leq r\leq R$.

\noindent Further let $m,L$ be positive integers, $K_{0}$ be a positive number
with $K \geq K_{0} >1$ and $f,g,W,\lambda,\mu$ to be defined as in Lemma
\ref{Functionestimate}. Finally, let $0<c<3$ and $M=\left\lceil cW^{1/3}%
\right\rceil $.

\noindent For $L\leq K$ and $1\leq m\leq\min\left\{  M,K-1\right\}  =M$ we
have%
\[
\left\vert \sum_{n=0}^{L-1-m}e\left(  g\left(  n\right)  \right)  \right\vert
\leq\frac{4\mu K}{\sqrt{\pi W}}m^{1/2}+\frac{\mu K}{W}m+4\sqrt{\frac{W}{\pi}%
}m^{-1/2}+1-\frac{6}{\pi}.
\]

\end{lemma}

\begin{remark}
Using the Riemann-Siegel bound, Lemma \ref{RiemannSiegel}, we later apply
Lemma \ref{SecondderivativePart2} to the initial range $t\geq3.8\cdot10^{7}$.
Consequently, the assumptions about the minimal bound for t in Lemma
\ref{SecondderivativePart2} are by no means restrictive. On the other hand, a
more general version of this Lemma could be obtained.
\end{remark}

\begin{proof}
\noindent We transfer the proof for these important Lemmas to Section \ref{Secondderivateproof}.
\end{proof}

\begin{lemma}
\label{EasySums} For any positive integer $M$ we have the following
inequalities:
\[%
\begin{array}
[c]{ll}%
\text{(a)} & {\displaystyle\sum\limits_{m=1}^{M}} \left(  1-\dfrac{m}%
{M}\right)  m^{-1/2}\leq\dfrac{4}{3}M^{1/2},\\
\text{(b)} & {\displaystyle\sum\limits_{m=1}^{M}} \left(  1-\dfrac{m}%
{M}\right)  =\dfrac{1}{2}\left(  M-1\right)  ,\\
\text{(c)} & {\displaystyle\sum\limits_{m=1}^{M}} \left(  1-\dfrac{m}%
{M}\right)  m^{1/2}\leq\dfrac{4}{15}M^{3/2},\\
\text{(d)} & {\displaystyle\sum\limits_{m=1}^{M}} \left(  1-\dfrac{m}%
{M}\right)  m\leq\dfrac{1}{6}M^{2}.
\end{array}
\]
\end{lemma}

\begin{proof}
\noindent Equation (b) and inequality (d) are easily established. The other inequalities can be found in (\cite{Hiary2}, p. 215). Compare also (\cite{ChengGraham}, Lemma 7).
\end{proof}

\begin{lemma}
\label{alphabetaEstimate} Let $t\geq5\cdot10^{6}$, $\phi\in\left[  1/3,
0.35\right]  $, $K=\left\lceil t^{\phi}\right\rceil $, $n_{1}=\left\lfloor
\sqrt{t/2\pi}\right\rfloor $, $R=\left\lfloor n_{1}/K \right\rfloor $ and
$4\leq r\leq R$.

\noindent Further let $m,L$ be positive integers, $K_{0}$ be a positive number
with $K \geq K_{0} >1$ and $f,g,W,\lambda,\mu$ to be defined as in Lemma
\ref{Functionestimate}. Finally, let $0<c<3$ and $M=\left\lceil cW^{1/3}%
\right\rceil $.

\noindent For $L\leq K$ we have
\[
\left\vert \sum_{n=0}^{L-1}e\left(  f\left(  n\right)  \right)  \right\vert
^{2}\leq\left(  \frac{K}{W^{1/3}}+c\right)  \left(  \alpha K+\beta
W^{2/3}\right)
\]
with%
\begin{align*}
\alpha &  =\frac{1}{c}+\frac{32\mu}{15\sqrt{\pi}}\sqrt{c+W^{-1/3}}+\frac{\mu
c}{3W^{1/3}}+\frac{\mu}{3W^{2/3}},\\
\beta &  =\frac{32}{3\sqrt{\pi c}}+\frac{1}{c}\left(  \frac{6}{\pi}-1\right)
\frac{1}{W^{2/3}}+\left(  1-\frac{6}{\pi}\right)  \frac{1}{W^{1/3}}.
\end{align*}
Additionally, we assert that $\alpha$ and $\beta$ always attain positive values.
\end{lemma}

\begin{proof}
From the definition of $S_{m}^{\prime}\left(  L\right)  $ in
Lemma \ref{Weyl} together with $N=-1$, Lemma \ref{SecondderivativePart2} and
Lemma \ref{EasySums} we establish%
\begin{align*}
&  \frac{2}{M}\sum_{m=1}^{M}\left(  1-\frac{m}{M}\right)  \left\vert
S_{m}^{\prime}\left(  L\right)  \right\vert \\
&  =\frac{2}{M}\sum_{m=1}^{M}\left(  1-\frac{m}{M}\right)  \left\vert
\sum_{n=0}^{L-1-m}e\left(  f\left(  m+n\right)  -f\left(  n\right)  \right)
\right\vert \\
&  =\frac{2}{M}\sum_{m=1}^{M}\left(  1-\frac{m}{M}\right)  \left\vert
\sum_{n=0}^{L-1-m}e\left(  g\left(  n\right)  \right)  \right\vert \\
&  \leq\frac{32\mu K}{15\sqrt{\pi W}}M^{\frac{1}{2}}+\frac{\mu K}{3W}%
M+\frac{32}{3}\sqrt{\frac{W}{\pi}}M^{-\frac{1}{2}}+\left(  1-\frac{1}{M}\right)  \left(
1-\frac{6}{\pi}\right)  .
\end{align*}
Now, applying Lemma \ref{Weyl} together with $M=\left\lceil cW^{1/3}\right\rceil$ we obtain
\begin{align*}
&  \left\vert \sum_{n=0}^{L-1}e\left(  f\left(  n\right)  \right)  \right\vert
^{2}\\
&  \leq\left(  K+M-1\right)  \left(  \frac{K}{M}+\frac{2}{M}\sum_{m=1}%
^{M}\left(  1-\frac{m}{M}\right)  \left\vert S_{m}^{\prime}\left(  L\right)
\right\vert \right)  \\
&  \leq\left(  K+M-1\right)  \left(  \frac{K}{M}+\frac{32\mu K}{15\sqrt{\pi
W}}M^{\frac{1}{2}}+\frac{\mu K}{3W}M\right.  \\
&  \qquad+\left.  \frac{32}{3}\sqrt{\frac{W}{\pi}}M^{-\frac{1}{2}}+\frac{1}{M}\left(
\frac{6}{\pi}-1\right)  +1-\frac{6}{\pi}\right)  \\
&  \leq\left(  K+cW^{\frac{1}{3}}\right)  \left(  \frac{K}{W^{\frac{1}{3}}%
}\left(  \frac{1}{c}+\frac{32\mu}{15\sqrt{\pi}}\sqrt{c+W^{-\frac{1}{3}}}%
+\frac{\mu c}{3W^{\frac{1}{3}}}+\frac{\mu}{3W^{\frac{2}{3}}}\right)  \right.
\\
&  \qquad+\left.  W^{\frac{1}{3}}\left(  \frac{32}{3\sqrt{\pi c}}+\frac
{1}{cW^{\frac{2}{3}}}\left(  \frac{6}{\pi}-1\right)  +\left(  1-\frac{6}{\pi
}\right)  \frac{1}{W^{\frac{1}{3}}}\right)  \right)  \\
&  =\left(  K+cW^{\frac{1}{3}}\right)  \left(  \frac{K}{W^{\frac{1}{3}}}%
\alpha+W^{\frac{1}{3}}\beta\right)  \\
&  =\left(  \frac{K}{W^{\frac{1}{3}}}+c\right)  \left(  \alpha K+\beta
W^{\frac{2}{3}}\right)  .
\end{align*}
Finally, we check the positivity for $\alpha$ and $\beta$. For $\alpha$ this is trivial. For $\beta$ we first estimate
\[
W^{\frac{1}{3}}=\frac{\pi^{\frac{1}{3}}\left(  r+1\right)  K}{t^{\frac{1}{3}}%
}\geq5\pi^{\frac{1}{3}}t^{\phi-\frac{1}{3}}\geq5\pi^{\frac{1}{3}}.
\]
Now, we look to
\begin{align*}
\beta W^{\frac{1}{3}}  & =\left(  \frac{32}{3\sqrt{\pi c}}+\frac{1}{c}\left(
\frac{6}{\pi}-1\right)  \frac{1}{W^{\frac{2}{3}}}+\left( 1-\frac{6}{\pi
}\right)  \frac{1}{W^{\frac{1}{3}}}\right)  W^{\frac{1}{3}}\\
& =\frac{32}{3\sqrt{\pi c}}W^{\frac{1}{3}}+\frac{1}{c}\left(  \frac{6}{\pi
}-1\right)  \frac{1}{W^{\frac{1}{3}}}+1-\frac{6}{\pi}\\
& \geq\frac{32}{3\sqrt{\pi c}}5\pi^{\frac{1}{3}}+1-\frac{6}{\pi} \geq \frac
{160}{\pi^{\frac{1}{6}}3\sqrt{3}}+1-\frac{6}{\pi}=24.533\ldots > 0.
\end{align*}
By the positivity for $W^{1/3}$ we obtain our desired inequality.
\end{proof}

\section{Proof for Theorem \ref{theorem1}}

\label{MainProof}

Proof for assertion (a): We split the proof into different regions for the parameter $t$. As we see
later, the van der Corput bounds (\ref{CorputSum1}) and (\ref{CorputSum2})
attain their full power starting at a very high range for $t$. Thus we have to
handle carefully the intermediate region between the Riemann-Siegel bound and
the van der Corput bounds. This bottleneck region requires a more detailed analysis.

\subsection{Proof for the region $3 \leq t \leq200$}

Here, we directly apply Lemma \ref{ComputationalBound} which gives the required bound
\[
\left\vert \zeta\left(  1/2+it\right)  \right\vert \leq0.611 t^{1/6}\log t.
\]
for the range $3\leq t\leq200$.

\subsection{Proof for the region $200\leq t\leq3.8\cdot10^{7}$}

Here, we apply the Riemann-Siegel bounds from Lemma \ref{RiemannSiegel}(b) for the range $200\leq t\leq3.3\cdot10^{7}$ and 
Lemma \ref{RiemannSiegel}(c) for the range $3.3\cdot10^{7} \leq t \leq 3.8\cdot 10^{7}$. We summarize
\[
\left\vert \zeta\left(  1/2+it\right)  \right\vert \leq0.611 t^{1/6}\log t.
\]
for the range $200\leq t \leq3.8\cdot 10^{7}$.

\subsection{Estimating the sums for $\left\vert \zeta\left(  1/2+it\right)
\right\vert $}

\label{Bottleneckregion} We start with parameters $3.8\cdot10^{7}\leq
t_{0}\leq t\leq t_{1}$, $\phi\in\left[  1/3, 0.35\right]  $, $K=\left\lceil
t^{\phi}\right\rceil $, $K_{0}=t_{0}^{\phi}$, $n_{1}=\left\lfloor \sqrt
{t/2\pi}\right\rfloor $, $R=\left\lfloor n_{1}/K\right\rfloor $ and
$4=r_{0}\leq r\leq R$. We also define%
\[
R_{0}=\left\lceil \frac{\sqrt{t_{0}/2\pi}-1}{t_{0}^{\phi}+1}-1\right\rceil
,R_{1}=\left\lfloor \frac{1}{\sqrt{2\pi}}t_{1}^{1/2-\phi}\right\rfloor
\]
and%
\[
\mu_{0}=\frac{1}{2}\left(  1+\frac{1}{r_{0}}\right)  \left(  2+\frac{1}{r_{0}%
}+\frac{1}{t_{0}^{\phi}-1}\right)  .
\]
Then, using the definition for $\mu$ from Lemma \ref{Functionestimate}, we
have%
\begin{align}
R  &  \leq\frac{1}{\sqrt{2\pi}}t^{1/2-\phi},\label{Intermediate1}\\
4  &  \leq R_{0}\leq R\leq R_{1},\label{Intermediate2}\\
\mu &  \leq\mu_{0}. \label{Intermediate3}%
\end{align}

\noindent Assertion (\ref{Intermediate1}) is easy to establish. For assertion
(\ref{Intermediate2}), the first inequalities $4\leq R_{0}\leq R$ follow from
the proof steps in Lemma \ref{SecondderivativePart2}(a). The remaining
inequality $R\leq R_{1}$ can be deduced with use of (\ref{Intermediate1}).
Assertion (\ref{Intermediate3}) is obtained by
\begin{align*}
\mu &  =\frac{1}{2}\lambda^{\frac{2}{3}}\left(  1+\frac{1}{\left(
1-K_{0}^{-1}\right)  \lambda^{\frac{1}{3}}}\right)  =\frac{1}{2}\lambda
^{\frac{1}{3}}\left(  \lambda^{\frac{1}{3}}+\frac{K_{0}}{K_{0}-1}\right) \\
&  =\frac{1}{2}\left(  1+\frac{1}{r}\right)  \left(  2+\frac{1}{r}+\frac
{1}{K_{0}-1}\right) \\
&  \leq\frac{1}{2}\left(  1+\frac{1}{r_{0}}\right)  \left(  2+\frac{1}{r_{0}%
}+\frac{1}{t_{0}^{\phi}-1}\right) \\
&  =\mu_{0}.
\end{align*}
In the next step we derive some important inequalities for the four relevant
terms in the main inequality in Lemma \ref{alphabetaEstimate}.

\medskip\noindent First, we have
\begin{equation}
\alpha\leq M_{1}:=\frac{1}{c}+\mu_{0}\left(  \frac{32}{15\pi^{\frac{1}{2}}%
}\sqrt{c+\frac{1}{5\pi^{\frac{1}{3}}t_{0}^{\phi-\frac{1}{3}}}}+\frac{c}%
{15\pi^{\frac{1}{3}}t_{0}^{\phi-\frac{1}{3}}}+\frac{1}{75\pi^{\frac{2}{3}%
}t_{0}^{2\left(  \phi-\frac{1}{3}\right)  }}\right)  . \label{Intermediate4}%
\end{equation}
Using%
\begin{equation}
W^{\frac{1}{3}}=\frac{\pi^{\frac{1}{3}}\left(  r+1\right)  K}{t^{\frac{1}{3}}%
}\geq\frac{\pi^{\frac{1}{3}}\left(  r_{0}+1\right)  \left\lceil t^{\phi
}\right\rceil }{t^{\frac{1}{3}}}\geq5\pi^{\frac{1}{3}}t^{\phi-\frac{1}{3}},
\label{Intermediate5}%
\end{equation}
we estimate%
\begin{align*}
\alpha &  =\frac{1}{c}+\mu\left(  \frac{32}{15\sqrt{\pi}}\sqrt{c+W^{-\frac
{1}{3}}}+\frac{c}{3W^{\frac{1}{3}}}+\frac{1}{3W^{\frac{2}{3}}}\right) \\
&  \leq\frac{1}{c}+\mu\left(  \frac{32}{15\sqrt{\pi}}\sqrt{c+\frac{1}%
{5\pi^{\frac{1}{3}}t^{\phi-\frac{1}{3}}}}+\frac{c}{15\pi^{\frac{1}{3}}%
t^{\phi-\frac{1}{3}}}+\frac{1}{75\pi^{\frac{2}{3}}t^{2\left(  \phi-\frac{1}{3}
\right)  }}\right)  .
\end{align*}
The assertion now follows from the conditions $t\geq t_{0}$ and
(\ref{Intermediate3}).

\medskip\noindent Second, we have
\begin{equation}
\label{Intermediate6}%
\begin{split}
\alpha c\frac{W^{\frac{1}{3}}}{K}  &  \leq M_{2}:=\frac{1}{t_{0}^{\phi
-\frac{1}{6}}}\frac{1}{\sqrt{2}\pi^{\frac{1}{6}}}\left(  1+\frac{1}{R_{0}%
}\right) \\
&  +\mu_{0}c\left(  \frac{16\sqrt{2}}{15\pi^{\frac{2}{3}}}\frac{1}{t_{0}%
^{\phi-\frac{1}{6}}}\left(  1+\frac{1}{R_{0}}\right)  \sqrt{c+\frac{1}%
{5\pi^{\frac{1}{3}}t_{0}^{\phi-\frac{1}{3}}}}+\frac{c}{3t_{0}^{\phi}}+\frac
{1}{15\pi^{\frac{1}{3}}t_{0}^{2\phi-\frac{1}{3}}}\right)  .
\end{split}
\end{equation}
Applying (\ref{Intermediate1}) give us the following result:
\begin{equation}
\frac{W^{\frac{1}{3}}}{K}=\frac{\pi^{\frac{1}{3}}\left(  r+1\right)
K}{t^{\frac{1}{3}}}\frac{1}{K}\leq\frac{\pi^{\frac{1}{3}}}{t^{\frac{1}{3}}%
}R\left(  1+\frac{1}{R}\right)  \leq\frac{1}{t^{\phi-\frac{1}{6}}}\frac
{1}{\sqrt{2}\pi^{\frac{1}{6}}}\left(  1+\frac{1}{R}\right)  .
\label{Intermediate7}%
\end{equation}
Now, using (\ref{Intermediate5}) and (\ref{Intermediate7}) we estimate%
\begin{align*}
\alpha c\frac{W^{\frac{1}{3}}}{K}  &  =\frac{W^{\frac{1}{3}}}{K}+\mu c\left(
\frac{32}{15\sqrt{\pi}}\sqrt{c+W^{-\frac{1}{3}}}\frac{W^{\frac{1}{3}}}%
{K}+\frac{c}{3K}+\frac{1}{3W^{\frac{1}{3}}K}\right) \\
&  \leq\frac{1}{t^{\phi-\frac{1}{6}}}\frac{1}{\sqrt{2}\pi^{\frac{1}{6}}%
}\left(  1+\frac{1}{R}\right) \\
&  +\mu c\left(  \frac{32}{15\sqrt{\pi}}\sqrt{c+\frac{1}{5\pi^{\frac{1}{3}%
}t^{\phi-\frac{1}{3}}}}\frac{1}{t^{\phi-\frac{1}{6}}}\frac{1}{\sqrt{2}%
\pi^{\frac{1}{6}}}\left(  1+\frac{1}{R}\right)  +\frac{c}{3t^{\phi}}+\frac
{1}{15\pi^{\frac{1}{3}}t^{2\phi-\frac{1}{3}}}\right)  .
\end{align*}
Together with $t\geq t_{0}$, (\ref{Intermediate2}) and (\ref{Intermediate3})
we obtain the assertion.

\medskip\noindent Third, we have
\begin{equation}
\label{Intermediate8}%
\begin{split}
\beta\frac{W^{\frac{2}{3}}}{K}  &  \leq M_{3}:=\frac{1}{t_{0}^{\phi}}\frac
{1}{c}\left(  \frac{6}{\pi}-1\right)  +\frac{16}{3\pi^{\frac{5}{6}}\sqrt{c}
}\frac{1}{t_{0}^{\phi-\frac{1}{3}}}\left(  1+\frac{1}{R_{0}}\right)
^{2}\left(  1+\frac{1}{t_{0}^{\phi}}\right) \\
&  +\left(  1-\frac{6}{\pi}\right)  \frac{5\pi^{\frac{1}{3}}}{t_{1}^{\frac
{1}{3}}}.
\end{split}
\end{equation}
Using (\ref{Intermediate1}) we estimate
\begin{equation}
\label{Intermediate9}%
\begin{split}
\frac{W^{\frac{2}{3}}}{K}  &  =\frac{\pi^{\frac{2}{3}}\left(  r+1\right)
^{2}K^{2}}{t^{\frac{2}{3}}}\frac{1}{K}\leq\frac{\pi^{\frac{2}{3}}\left(
R+1\right)  ^{2}K}{t^{\frac{2}{3}}}\\
&  \leq\frac{\pi^{\frac{2}{3}}\frac{1}{2\pi}t^{1-2\phi}}{t^{\frac{2}{3}}
}\left(  1+\frac{1}{R}\right)  ^{2}t^{\phi}\left(  1+\frac{1}{t^{\phi}}\right)
\\
&  =\frac{1}{t^{\phi-\frac{1}{3}}}\frac{1}{2\pi^{\frac{1}{3}}}\left(
1+\frac{1}{R}\right)  ^{2}\left(  1+\frac{1}{t^{\phi}}\right)  .
\end{split}
\end{equation}
Further, we have%
\begin{equation}
\frac{W^{\frac{1}{3}}}{K}=\frac{\pi^{\frac{1}{3}}\left(  r+1\right)
K}{t^{\frac{1}{3}}}\frac{1}{K}\geq\frac{\pi^{\frac{1}{3}}\left(
r_{0}+1\right)  }{t^{\frac{1}{3}}}=\frac{5\pi^{\frac{1}{3}}}{t^{\frac{1}{3}}}.
\label{Intermediate10}%
\end{equation}
From the definition of $\beta$ in Lemma \ref{alphabetaEstimate} and using
(\ref{Intermediate9}), (\ref{Intermediate10}) we obtain
\begin{equation}
\label{Intermediate11}%
\begin{split}
\beta\frac{W^{\frac{2}{3}}}{K}  &  =\left(  \frac{32}{3\sqrt{\pi c}}+\frac
{1}{c}\left(  \frac{6}{\pi}-1\right)  \frac{1}{W^{\frac{2}{3}}}+\left(
1-\frac{6}{\pi}\right)  \frac{1}{W^{\frac{1}{3}}}\right)  \frac{W^{\frac{2}
{3}}}{K}\\
&  =\frac{32}{3\sqrt{\pi c}}\frac{W^{\frac{2}{3}}}{K}+\frac{1}{c}\left(
\frac{6}{\pi}-1\right)  \frac{1}{K}+\left(  1-\frac{6}{\pi}\right)
\frac{W^{\frac{1}{3}}}{K}\\
&  \leq\frac{32}{3\sqrt{\pi c}}\frac{1}{t^{\phi-\frac{1}{3}}}\frac{1}%
{2\pi^{\frac{1}{3}}}\left(  1+\frac{1}{R}\right)  ^{2}\left(  1+\frac
{1}{t^{\phi}}\right)  +\frac{1}{c}\left(  \frac{6}{\pi}-1\right)  \frac
{1}{t^{\phi}}+\left(  1-\frac{6}{\pi}\right)  \frac{5\pi^{\frac{1}{3}}%
}{t^{\frac{1}{3}}}\\
&  =\frac{1}{t^{\phi}}\frac{1}{c}\left(  \frac{6}{\pi}-1\right)  +\frac
{16}{3\sqrt{c}\pi^{\frac{5}{6}}}\frac{1}{t^{\phi-\frac{1}{3}}}\left(  1+\frac
{1}{R}\right)  ^{2}\left(  1+\frac{1}{t^{\phi}}\right)  +\left(  1-\frac
{6}{\pi}\right)  \frac{5\pi^{\frac{1}{3}}}{t^{\frac{1}{3}}}.
\end{split}
\end{equation}
Together with $t_{0}\leq t\leq t_{1}$ and (\ref{Intermediate2}) we obtain our
desired inequality.

\medskip\noindent Our last assertion reads
\begin{equation}
\label{Intermediate12}%
\begin{split}
c\beta\frac{W}{K^{2}}  &  \leq M_{4}:=\frac{1}{t_{0}^{2\phi-\frac{1}{6}}}
\frac{1}{\sqrt{2}\pi^{\frac{1}{6}}}\left(  \frac{6}{\pi}-1\right)  \left(
1+\frac{1}{R_{0}}\right) \\
&  +\frac{1}{t_{0}^{2\phi-\frac{1}{2}}}\frac{8\sqrt{2c}}{3\pi}\left(
1+\frac{1}{R_{0}}\right)  ^{3}\left(  1+\frac{1}{t_{0}^{\phi}}\right) \\
&  +\left(  1-\frac{6}{\pi}\right)  \frac{1}{t_{1}^{\phi+\frac{1}{6}}}
\frac{5\pi^{\frac{1}{6}}c}{\sqrt{2}}\left(  1+\frac{1}{R_{1}}\right)  .
\end{split}
\end{equation}
Using (\ref{Intermediate7}), (\ref{Intermediate11}) and expanding%
\[
c\beta\frac{W}{K^{2}}=\left(  \beta\frac{W^{\frac{2}{3}}}{K}\right)
\left(c \frac{W^{\frac{1}{3}}}{K} \right)
\]
we obtain the assertion together with $t_{0}\leq t\leq t_{1}$ and
(\ref{Intermediate2}).

\medskip\noindent Now, we start estimating our main sum. Following the proof
in Lemma \ref{RiemannSiegel}, we first obtain%
\begin{equation}
\label{MainSumEstimate}%
\begin{split}
\left\vert \zeta\left(  1/2+it\right)  \right\vert  &  \leq2\left\vert
\sum\limits_{n=1}^{n_{1}}n^{-\frac{1}{2}+it}\right\vert +\cos\frac{\pi}
{8}\left(  2\pi\right)  ^{\frac{1}{4}}t^{-\frac{1}{4}}+0.127t^{-\frac{3}{4}}\\
&  \leq2\left\vert \sum\limits_{n=1}^{n_{1}}n^{-\frac{1}{2}+it}\right\vert
+1.463t^{-\frac{1}{4}}+0.127t^{-\frac{3}{4}}.
\end{split}
\end{equation}
We will now brake the main sum into three parts by
\begin{equation}
\label{MainSumBrake}%
\begin{split}
2\left\vert \sum\limits_{n=1}^{n_{1}}n^{-\frac{1}{2}+it}\right\vert  &
\leq2\left\vert \sum\limits_{n=1}^{Kr_{0}-1}n^{-\frac{1}{2}+it}\right\vert
+2\left\vert \sum\limits_{n=Kr_{0}}^{n_{1}}n^{-\frac{1}{2}+it}\right\vert \\
&  \leq2\sum\limits_{n=1}^{Kr_{0}-1}n^{-\frac{1}{2}}+2\sum_{r=r_{0}}
^{R-1}\left\vert \sum\limits_{n=Kr}^{K\left(  r+1\right)  -1}n^{-\frac{1}
{2}+it}\right\vert +2\left\vert \sum\limits_{n=KR}^{n_{1}}n^{-\frac{1}{2}%
+it}\right\vert \\
&  =T_{1}+T_{2}+T_{3}.
\end{split}
\end{equation}
We start with $T_{2}$. By applying partial summation together with Abel's
inequality and Lemma \ref{alphabetaEstimate} we arrive at%
\begin{align*}
\left\vert \sum\limits_{Kr\leq n<K\left(  r+1\right)  }^{{}}n^{-\frac{1}%
{2}+it}\right\vert  &  =\left\vert \sum\limits_{n=0}^{K-1}\frac{e^{it\log
\left(  Kr+n\right)  }}{\sqrt{Kr+n}}\right\vert \leq\frac{1}{\sqrt{Kr}}%
\max_{1\leq L\leq K}\left\vert \sum\limits_{n=0}^{L-1}e^{it\log\left(
Kr+n\right)  }\right\vert \\
&  =\frac{1}{\sqrt{Kr}}\max_{1\leq L\leq K}\left\vert \sum\limits_{n=0}%
^{L-1}e\left(  f\left(  n\right)  \right)  \right\vert \\
&  \leq\frac{1}{\sqrt{Kr}}\sqrt{\left(  \frac{K}{W^{\frac{1}{3}}}+c\right)
\left(  \alpha K+\beta W^{\frac{2}{3}}\right)  }.
\end{align*}
We summarize%
\begin{equation}
T_{2}\leq2\sum_{r=r_{0}}^{R-1}\frac{1}{\sqrt{Kr}}\sqrt{\left(  \frac
{K}{W^{\frac{1}{3}}}+c\right)  \left(  \alpha K+\beta W^{\frac{2}{3}}\right)
}. \label{T2Estimate}%
\end{equation}
Similar, for $T_{3}$ we get%
\[
T_{3}\leq\frac{2}{\sqrt{KR}}\max_{1\leq L\leq n_{1}-KR+1}\left\vert
\sum\limits_{n=0}^{L-1}e^{it\log\left(  KR+n\right)  }\right\vert .
\]
Observing that $R=\left\lfloor n_{1}/K\right\rfloor >n_{1}/K-1$ we first
estimate $n_{1}\leq\left(  R+1\right)  K-1$. Putting together we arrive at%
\begin{align*}
T_{3}  &  \leq\frac{2}{\sqrt{KR}}\max_{1\leq L\leq\left(  R+1\right)
K-1-KR+1}\left\vert \sum\limits_{n=0}^{L-1}e^{it\log\left(  KR+n\right)
}\right\vert \\
&  =\frac{2}{\sqrt{KR}}\max_{1\leq L\leq K}\left\vert \sum\limits_{n=0}%
^{L-1}e\left(  f\left(  n\right)  \right)  \right\vert \\
&  \leq\frac{2}{\sqrt{KR}}\sqrt{\left(  \frac{K}{W^{\frac{1}{3}}}+c\right)
\left(  \alpha K+\beta W^{\frac{2}{3}}\right)  }.
\end{align*}
Combining together with (\ref{T2Estimate}) and applying some standard
manipulations we finally arrive at
\begin{equation}
\label{T2T3Estimate}%
\begin{split}
T_{2}+T_{3}  &  \leq2\sum_{r=r_{0}}^{R}\frac{1}{\sqrt{Kr}}\sqrt{\left(
\frac{K}{W^{\frac{1}{3}}}+c\right)  \left(  \alpha K+\beta W^{\frac{2}{3}%
}\right)  }\\
&  =t^{\frac{1}{6}}\sum_{r=r_{0}}^{R}\frac{1}{\sqrt{r\left(  r+1\right)  }
}\sqrt{\frac{4}{\pi^{\frac{1}{3}}}\left(  \alpha+\frac{\alpha cW^{\frac{1}{3}
}}{K}+\frac{\beta W^{\frac{2}{3}}}{K}+\frac{c\beta W}{K^{2}}\right)  }.
\end{split}
\end{equation}
From $Kr_{0}-1=\left\lceil t^{\phi}\right\rceil r_{0}-1\geq$ $t^{\phi}%
r_{0}-1\geq t_{0}^{\phi}r_{0}-1$ and since $Kr_{0}-1$ is an integer it follows
that $Kr_{0}-1\geq\left\lceil t_{0}^{\phi}r_{0}-1\right\rceil =\left\lceil
t_{0}^{\phi}r_{0}\right\rceil -1.$ Thus, we can split the $T_{1}$ sum into two
parts and estimate further to
\begin{align*}
T_{1}  &  =2\sum\limits_{n=1}^{Kr_{0}-1}\frac{1}{\sqrt{n}}=2\sum
\limits_{n=1}^{\left\lceil t_{0}^{\phi}r_{0}\right\rceil -1}\frac{1}{\sqrt{n}%
}+2\sum\limits_{n=\left\lceil t_{0}^{\phi}r_{0}\right\rceil }^{Kr_{0}-1}%
\frac{1}{\sqrt{n}}\\
&  \leq2\sum\limits_{n=1}^{\left\lceil t_{0}^{\phi}r_{0}\right\rceil -1}%
\frac{1}{\sqrt{n}}+2\int\limits_{\left\lceil t_{0}^{\phi}r_{0}\right\rceil
-1}^{Kr_{0}-1}\frac{1}{\sqrt{x}}dx\\
&  =2\sum\limits_{n=1}^{\left\lceil t_{0}^{\phi}r_{0}\right\rceil -1}\frac
{1}{\sqrt{n}}+4\sqrt{Kr_{0}-1}-4\sqrt{\left\lceil t_{0}^{\phi}r_{0}%
\right\rceil -1}.
\end{align*}
Now, a standard argument on the inner term establishes
\begin{equation}
T_{1}\leq t^{\frac{\phi}{2}}4\sqrt{r_{0}\left(  1+\frac{1}{t_{0}^{\phi}%
}\right)  }+2\sum\limits_{n=1}^{\left\lceil t_{0}^{\phi}r_{0}\right\rceil
-1}\frac{1}{\sqrt{n}}-4\sqrt{\left\lceil t_{0}^{\phi}r_{0}\right\rceil -1}.
\label{T1Estimate}%
\end{equation}
Putting together (\ref{MainSumEstimate}), (\ref{MainSumBrake}),
(\ref{T2T3Estimate}) and (\ref{T1Estimate}) we obtain
\begin{equation}
\label{ZetaEstimate}%
\begin{split}
\left\vert \zeta\left(  \frac{1}{2}+it\right)  \right\vert  &  \leq
t^{\frac{1}{6}}\sum_{r=r_{0}}^{R}\frac{1}{\sqrt{r\left(  r+1\right)  }}
\sqrt{\frac{4}{\pi^{\frac{1}{3}}}\left(  \alpha+\frac{\alpha cW^{\frac{1}{3}}
}{K}+\frac{\beta W^{\frac{2}{3}}}{K}+\frac{c\beta W}{K^{2}}\right)  }\\
&  +t^{\frac{\phi}{2}}4\sqrt{r_{0}\left(  1+\frac{1}{t_{0}^{\phi}}\right)
}+2\sum\limits_{n=1}^{\left\lceil t_{0}^{\phi}r_{0}\right\rceil -1}\frac
{1}{\sqrt{n}}-4\sqrt{\left\lceil t_{0}^{\phi}r_{0}\right\rceil -1}\\
&  +1.463t^{-\frac{1}{4}}+0.127t^{-\frac{3}{4}}.
\end{split}
\end{equation}
Next, we apply the Jensen inequality and observe that
\begin{align*}
\sum_{r=r_{0}}^{R}\frac{1}{\sqrt{r\left(  r+1\right)  }}  &  \leq\frac
{1}{\sqrt{r_{0}\left(  r_{0}+1\right)  }}+\sum_{r=r_{0}+1}^{R}\int_{r-\frac
{1}{2}}^{r+\frac{1}{2}}\frac{dx}{\sqrt{x\left(  x+1\right)  }}\\
&  =\frac{1}{\sqrt{r_{0}\left(  r_{0}+1\right)  }}+\int_{r_{0}+\frac{1}{2}%
}^{R+\frac{1}{2}}\frac{dx}{\sqrt{x\left(  x+1\right)  }}\\
&  =\frac{1}{\sqrt{r_{0}\left(  r_{0}+1\right)  }}+2\sinh^{-1}\sqrt{x}%
\mid_{r_{0}+\frac{1}{2}}^{R+\frac{1}{2}}.
\end{align*}
Together with $\sinh^{-1}\sqrt{x}=\log\left(  \sqrt{x}+\sqrt{x+1}\right)  $,
(\ref{Intermediate1}), (\ref{Intermediate2}), by some easy manipulations, we
arrive at
\begin{equation}
\label{SumIntegralEstimate}%
\begin{split}
&  \sum_{r=r_{0}}^{R}\frac{1}{\sqrt{r\left(  r+1\right)  }}\\
&  \leq\left(  \frac{1}{2}-\phi\right)  \log t+\log\left(  1+\frac{1}{2R_{0}
}\right)  +2\log\left(  1+\sqrt{1+\frac{1}{R_{0}+\frac{1}{2}}}\right) \\
&  +\frac{1}{\sqrt{r_{0}\left(  r_{0}+1\right)  }}-\log\sqrt{2\pi}-2\sinh
^{-1}\sqrt{r_{0}+\frac{1}{2}}.
\end{split}
\end{equation}
Now, we recall the constants $M_{1}$, $M_{2}$, $M_{3}$ and $M_{4}$ from
(\ref{Intermediate4}), (\ref{Intermediate6}), (\ref{Intermediate8}),
(\ref{Intermediate12}) and we continue with the definition of the following
new constants
\begin{align*}
M_{5}  &  :=\log\left(  1+\frac{1}{2R_{0}}\right)  +2\log\left(
1+\sqrt{1+\frac{1}{R_{0}+\frac{1}{2}}}\right)  +\frac{1}{\sqrt{r_{0}\left(
r_{0}+1\right)  }}\\
&  \qquad-\log\sqrt{2\pi}-2\sinh^{-1}\sqrt{r_{0}+\frac{1}{2}},\\
M_{6}  &  :=4\sqrt{r_{0}\left(  1+\frac{1}{t_{0}^{\phi}}\right)  },\\
M_{7}  &  :=2\sum\limits_{n=1}^{\left\lceil t_{0}^{\phi}r_{0}\right\rceil
-1}\frac{1}{\sqrt{n}}-4\sqrt{\left\lceil t_{0}^{\phi}r_{0}\right\rceil -1}.
\end{align*}
Then, recalling (\ref{ZetaEstimate}) and (\ref{SumIntegralEstimate}), we
obtain%
\begin{equation}
\sum_{r=r_{0}}^{R}\frac{1}{\sqrt{r\left(  r+1\right)  }}\leq\left(  \frac
{1}{2}-\phi\right)  \log t+M_{5} \label{SquareSumEstimate}%
\end{equation}
and therefore
\begin{equation}
\label{GeneralFormula}%
\begin{split}
\left\vert \zeta\left(  \frac{1}{2}+it\right)  \right\vert  &  \leq
t^{\frac{1}{6}}\left(  \left(  \frac{1}{2}-\phi\right)  \log t+M_{5}\right)
\sqrt{\frac{4}{\pi^{\frac{1}{3}}}\left(  \alpha+\frac{\alpha cW^{\frac{1}{3}}
}{K}+\frac{\beta W^{\frac{2}{3}}}{K}+\frac{c\beta W}{K^{2}}\right)  }\\
&  \qquad+t^{\frac{\phi}{2}}M_{6}+M_{7}+1.463t^{-\frac{1}{4}}+0.127t^{-\frac
{3}{4}}.
\end{split}
\end{equation}
Finally, recalling (\ref{Intermediate4}), (\ref{Intermediate6}),
(\ref{Intermediate8}) and (\ref{Intermediate12}), we observe that%
\begin{align*}
\alpha &  \leq M_{1},\qquad t_{0}\leq t,\\
\alpha c\frac{W^{\frac{1}{3}}}{K}  &  \leq M_{2},\qquad t_{0}\leq t,\\
\beta\frac{W^{\frac{2}{3}}}{K}  &  \leq M_{3},\qquad t_{0}\leq t\leq t_{1},\\
c\beta\frac{W}{K^{2}}  &  \leq M_{4},\qquad t_{0}\leq t\leq t_{1}.
\end{align*}
Combining together with (\ref{GeneralFormula}) we finally arrive at%
\begin{equation}
\left\vert \zeta\left(  \frac{1}{2}+it\right)  \right\vert \leq a_{1}%
t^{\frac{1}{6}}\log t+a_{2}t^{\frac{1}{6}}+a_{3}t^{\frac{\phi}{2}}+a_{4}
\label{CorputSum1}%
\end{equation}
with the constants
\begin{equation}
\label{Coefficients1}%
\begin{split}
a_{1}  &  =\left(  \frac{1}{2}-\phi\right)  \sqrt{\frac{4}{\pi^{\frac{1}{3}}
}\left(  M_{1}+M_{2}+M_{3}+M_{4}\right)  },\\
a_{2}  &  =M_{5}\sqrt{\frac{4}{\pi^{\frac{1}{3}}}\left(  M_{1}+M_{2}
+M_{3}+M_{4}\right)  },\\
a_{3}  &  =M_{6},\\
a_{4}  &  =M_{7}+1.463t_{0}^{-\frac{1}{4}}+0.127t_{0}^{-\frac{3}{4}}.
\end{split}
\end{equation}

\subsection{Proof for the region $3.8\cdot10^{7}\leq t\leq8.97\cdot10^{17}$}

\label{ComputationalBottleneck}

Now, we will proceed through the intermediate bottleneck region which extends over the long
distance range
\[
3.80\cdot10^{7}\leq t\leq8.97\cdot10^{17}.
\]
Here, we split this region into $10$ subregions, starting at a certain value
$t_{0}$ and ending at $t_{1}$. As we will see later, the coefficients
(\ref{Coefficients1}) for the van der Corput bound (\ref{CorputSum1}) are
highly sensitive with respect to these parameters, especially at the beginning
of this process at $t_{0}=3.8\cdot10^{7}$.

\subsubsection{Bottleneck region $3.80\cdot10^{7} \leq t \leq3.85 \cdot10^{7}%
$}

Together with (\ref{CorputSum1}), (\ref{Coefficients1}) and selecting
$t_{0}=3.80\cdot10^{7}$, $t_{1} = 3.85\cdot10^{7}$, $c=1.82$ and $\phi=0.33559$
we obtain
\begin{align*}
\left\vert \zeta\left(  1/2+it\right)  \right\vert  &  \leq 0.622447t^{1/6}\log
t-8.21438t^{1/6}+8.01143t^{\phi/2}-2.87533\\
&  \leq0.611 t^{1/6}\log t,\text{ for }t_{0}\leq t\leq t_{1}.
\end{align*}

\subsubsection{Bottleneck region $3.85\cdot10^{7} \leq t \leq3.96 \cdot10^{7}%
$}

Together with (\ref{CorputSum1}), (\ref{Coefficients1}) and selecting
$t_{0}=3.85\cdot10^{7}$, $t_{1} = 3.96\cdot10^{7}$, $c=1.82$ and $\phi=0.33576$
we obtain%
\begin{align*}
\left\vert \zeta\left(  1/2+it\right)  \right\vert  &  \leq 0.621314t^{1/6}\log
t-8.20793t^{1/6}+8.01135t^{\phi/2}-2.8755\\
&  \leq0.611 t^{1/6}\log t,\text{ for }t_{0}\leq t\leq t_{1}.
\end{align*}

\subsubsection{Bottleneck region $3.96\cdot10^{7} \leq t \leq4.12 \cdot10^{7}%
$}

Together with (\ref{CorputSum1}), (\ref{Coefficients1}) and selecting
$t_{0}=3.96\cdot10^{7}$, $t_{1} = 4.12\cdot10^{7}$, $c=1.82$ and $\phi=0.33599$
we obtain%
\begin{align*}
\left\vert \zeta\left(  1/2+it\right)  \right\vert  &  \leq 0.619747t^{1/6}\log
t-8.19871t^{1/6}+8.01119t^{\phi/2}-2.8758\\
&  \leq0.611 t^{1/6}\log t,\text{ for }t_{0}\leq t\leq t_{1}.
\end{align*}

\subsubsection{Bottleneck region $4.12\cdot10^{7} \leq t \leq4.43 \cdot10^{7}%
$}

Together with (\ref{CorputSum1}), (\ref{Coefficients1}) and selecting
$t_{0}=4.12\cdot10^{7}$, $t_{1} = 4.43\cdot10^{7}$, $c=1.82$ and $\phi=0.33643$
we obtain%
\begin{align*}
\left\vert \zeta\left(  1/2+it\right)  \right\vert  &  \leq 0.61681t^{1/6}\log
t-8.1818t^{1/6}+8.01096t^{\phi/2}-2.87626\\
&  \leq0.611 t^{1/6}\log t,\text{ for }t_{0}\leq t\leq t_{1}.
\end{align*}

\subsubsection{Bottleneck region $4.43\cdot10^{7} \leq t \leq4.96 \cdot10^{7}%
$}

Together with (\ref{CorputSum1}), (\ref{Coefficients1}) and selecting
$t_{0}=4.43\cdot10^{7}$, $t_{1} = 4.96\cdot10^{7}$, $c=1.82$ and $\phi=0.33711$
we obtain%
\begin{align*}
\left\vert \zeta\left(  1/2+it\right)  \right\vert  &  \leq 0.612254t^{1/6}\log
t-8.15526t^{1/6}+8.01057t^{\phi/2}-2.87706\\
&  \leq0.611 t^{1/6}\log t,\text{ for }t_{0}\leq t\leq t_{1}.
\end{align*}

\subsubsection{Bottleneck region $4.96\cdot10^{7} \leq t \leq5.82 \cdot10^{7}%
$}

Together with (\ref{CorputSum1}), (\ref{Coefficients1}) and selecting
$t_{0}=4.96\cdot10^{7}$, $t_{1} = 5.82\cdot10^{7}$, $c=1.81$ and $\phi=0.33816$
we obtain%
\begin{align*}
\left\vert \zeta\left(  1/2+it\right)  \right\vert  &  \leq 0.605265t^{1/6}\log
t-8.11448t^{1/6}+8.00999t^{\phi/2}-2.87828\\
&  \leq0.611 t^{1/6}\log t,\text{ for }t_{0}\leq t\leq t_{1}.
\end{align*}

\subsubsection{Bottleneck region $5.82\cdot10^{7} \leq t \leq7.06 \cdot10^{7}%
$}

Together with (\ref{CorputSum1}), (\ref{Coefficients1}) and selecting
$t_{0}=5.82\cdot10^{7}$, $t_{1} = 7.06\cdot10^{7}$, $c=1.81$ and $\phi=0.33962$
we obtain%
\begin{align*}
\left\vert \zeta\left(  1/2+it\right)  \right\vert  &  \leq 0.595642t^{1/6}\log
t-8.05816t^{1/6}+8.00922t^{\phi/2}-2.87995\\
&  \leq0.611 t^{1/6}\log t,\text{ for }t_{0}\leq t\leq t_{1}.
\end{align*}

\subsubsection{Bottleneck region $7.06\cdot10^{7} \leq t \leq8.24 \cdot10^{7}%
$}

Together with (\ref{CorputSum1}), (\ref{Coefficients1}) and selecting
$t_{0}=7.06\cdot10^{7}$, $t_{1} = 8.24\cdot10^{7}$, $c=1.80$ and $\phi=0.34098$
we obtain%
\begin{align*}
\left\vert \zeta\left(  1/2+it\right)  \right\vert  &  \leq 0.58663t^{1/6}\log
t-8.00412t^{1/6}+8.00842t^{\phi/2}-2.8818\\
&  \leq0.611 t^{1/6}\log t,\text{ for }t_{0}\leq t\leq t_{1}.
\end{align*}

\subsubsection{Bottleneck region $8.24\cdot10^{7} \leq t \leq3.92 \cdot10^{8}%
$}

Together with (\ref{CorputSum1}), (\ref{Coefficients1}) and selecting
$t_{0}=8.24\cdot10^{7}$, $t_{1} = 3.92\cdot10^{8}$, $c=1.83$ and $\phi=0.33536$
we obtain%
\begin{align*}
\left\vert \zeta\left(  1/2+it\right)  \right\vert  &  \leq 0.617933t^{1/6}\log
t-8.20176t^{1/6}+8.00885t^{\phi/2}-2.88182\\
&  \leq0.611 t^{1/6}\log t,\text{ for }t_{0}\leq t\leq t_{1}.
\end{align*}

\subsubsection{Bottleneck region $3.92\cdot10^{8} \leq t \leq8.97 \cdot10^{17}%
$}
\label{LastBottleneck}

Together with (\ref{CorputSum1}), (\ref{Coefficients1}) and selecting
$t_{0}=3.92\cdot10^{8}$, $t_{1} = 8.97\cdot10^{17}$, $c=1.83$ and $\phi=0.33711$
we obtain%
\begin{align*}
\left\vert \zeta\left(  1/2+it\right)  \right\vert  &  \leq 0.597618t^{1/6}\log
t-8.09885t^{1/6}+8.00507t^{\phi/2}-2.89251\\
&  \leq0.611 t^{1/6}\log t,\text{ for }t_{0}\leq t\leq t_{1}.
\end{align*}

\subsection{Proof for the region $8.97\cdot10^{17}\leq t$}

Here, we use a similar method as in subsection \ref{Bottleneckregion}. We
start with parameters $8.97\cdot10^{17}=t_{0}\leq t$, $\phi=1/3$ together with
our previous parameters $K=\left\lceil t^{\phi}\right\rceil $, $K_{0}%
=t_{0}^{\phi}$, $n_{1}=\left\lfloor \sqrt{t/2\pi}\right\rfloor $,
$R=\left\lfloor n_{1}/K\right\rfloor $, $4=r_{0}\leq r\leq R$. We also recall
$R_{0}$ from subsection \ref{Bottleneckregion}.

\medskip\noindent Before we begin, we should observe from subsection
\ref{Bottleneckregion} that the assertions (\ref{Intermediate1}), $4\leq
R_{0}\leq R$ from (\ref{Intermediate2}), (\ref{Intermediate4}),
(\ref{Intermediate6}), (\ref{MainSumEstimate}), (\ref{MainSumBrake}),
(\ref{T2T3Estimate}), (\ref{ZetaEstimate}) and (\ref{SquareSumEstimate}) are
still valid without the former restriction $t\leq t_{1}$. Then, we estimate
\begin{equation}
\label{Intermediate24}%
\begin{split}
\frac{1}{r\left(  r+1\right)  }\frac{W^{\frac{2}{3}}}{K}  &  =\frac{\pi
^{\frac{2}{3}}\left(  r+1\right)  ^{2}K^{2}}{r\left(  r+1\right)  t^{\frac
{2}{3}}K}=\frac{\pi^{\frac{2}{3}}\left(  r+1\right)  K}{rt^{\frac{2}{3}}}\\
&  \leq\frac{\pi^{\frac{2}{3}}}{t^{\frac{2}{3}}}\left(  1+\frac{1}{r}\right)
t^{\frac{1}{3}}\left(  1+\frac{1}{t^{\frac{1}{3}}}\right) \\
&  \leq\frac{\pi^{\frac{2}{3}}}{t^{\frac{1}{3}}}\left(  1+\frac{1}{r_{0}%
}\right)  \left(  1+\frac{1}{t_{0}^{\frac{1}{3}}}\right)  .
\end{split}
\end{equation}
Similarly, we estimate%
\begin{equation}
\label{Intermediate25}%
\begin{split}
\frac{1}{r\left(  r+1\right)  }\frac{W}{K^{2}}  &  =\frac{\pi\left(
r+1\right)  ^{3}K^{3}}{r\left(  r+1\right)  tK^{2}}=\pi r\left(  1+\frac{1}
{r}\right)  ^{2}\frac{K}{t}\\
&  \leq\pi R\left(  1+\frac{1}{r}\right)  ^{2}\frac{K}{t}\leq\pi n_{1}\left(
1+\frac{1}{r}\right)  ^{2}\frac{1}{t}\\
&  \leq\frac{\pi^{\frac{1}{2}}}{t^{\frac{1}{3}}\sqrt{2}}\left(  1+\frac
{1}{r_{0}}\right)  ^{2}\frac{1}{t_{0}^{\frac{1}{6}}}.
\end{split}
\end{equation}
A standard argument gives us%
\begin{equation}
W^{\frac{2}{3}}\geq\left(  1+r_{0}\right)  ^{2}\pi^{\frac{2}{3}}
\label{Intermediate26}%
\end{equation}
and%
\begin{equation}
\beta\leq\frac{32}{3\sqrt{\pi c}}+\frac{1}{c}\left(  \frac{6}{\pi}-1\right)
\frac{1}{W^{\frac{2}{3}}}. \label{Intermediate27}%
\end{equation}
Now, recalling the $T_{2}$ and $T_{3}$ sums from (\ref{MainSumBrake}) together
with (\ref{T2T3Estimate}), we estimate
\begin{equation}
\label{T2T3T4T5}%
\begin{split}
T_{2}+T_{3}  &  \leq t^{\frac{1}{6}}\sum_{r=r_{0}}^{R}\frac{1}{\sqrt{r\left(
r+1\right)  }}\sqrt{\frac{4}{\pi^{\frac{1}{3}}}\left(  \alpha+\frac{\alpha
cW^{\frac{1}{3}}}{K}+\frac{\beta W^{\frac{2}{3}}}{K}+\frac{c\beta W}{K^{2}
}\right)  }\\
&  \leq t^{\frac{1}{6}}\sum_{r=r_{0}}^{R}\frac{1}{\sqrt{r\left(  r+1\right)
}}\sqrt{\frac{4}{\pi^{\frac{1}{3}}}\left(  \alpha+\frac{\alpha cW^{\frac{1}
{3}}}{K}\right)  }\\
&  +t^{\frac{1}{6}}\sum_{r=r_{0}}^{R}\sqrt{\frac{4}{\pi^{\frac{1}{3}}}\left(
\frac{1}{r\left(  r+1\right)  }\frac{\beta W^{\frac{2}{3}}}{K}+\frac
{1}{r\left(  r+1\right)  }\frac{c\beta W}{K^{2}}\right)  }\\
&  =T_{4}+T_{5}.
\end{split}
\end{equation}
We start with the $T_{4}$ sum. Recalling (3.4), (3.6) and (3.20) we estimate
\begin{equation}
\label{T4Estimate}%
\begin{split}
T_{4}  &  \leq t^{\frac{1}{6}}\sqrt{\frac{4}{\pi^{\frac{1}{3}}}\left(
M_{1}+M_{2}\right)  }\sum_{r=r_{0}}^{R}\frac{1}{\sqrt{r\left(  r+1\right)  }
}\\
&  \leq t^{\frac{1}{6}}\left(  \frac{1}{6}\log t+M_{5}\right)  \sqrt{\frac
{4}{\pi^{\frac{1}{3}}}\left(  M_{1}+M_{2}\right)  }\\
&  =t^{\frac{1}{6}}\log t\frac{1}{3\pi^{\frac{1}{6}}}\sqrt{M_{1}+M_{2}
}+t^{\frac{1}{6}}\frac{2}{\pi^{\frac{1}{6}}}M_{5}\sqrt{M_{1}+M_{2}}.
\end{split}
\end{equation}
For the remaining sum $T_{5}$ we operate as follows. Recalling
(\ref{Intermediate1}), (\ref{Intermediate24}), (\ref{Intermediate25}),
(\ref{Intermediate26}) and (\ref{Intermediate27}), we obtain
\begin{equation}
\label{T5Estimate}%
\begin{split}
T_{5}  &  =t^{\frac{1}{6}}\sum_{r=r_{0}}^{R}\sqrt{\frac{4}{\pi^{\frac{1}{3}}
}\left(  \frac{1}{r\left(  r+1\right)  }\frac{\beta W^{\frac{2}{3}}}{K}
+\frac{1}{r\left(  r+1\right)  }\frac{c\beta W}{K^{2}}\right)  }\\
&  \leq\sqrt{\frac{4\beta}{\pi^{\frac{1}{3}}}\left(  \pi^{\frac{2}{3}}\left(
1+\frac{1}{r_{0}}\right)  \left(  1+\frac{1}{t_{0}^{\frac{1}{3}}}\right)
+\frac{c\pi^{\frac{1}{2}}}{\sqrt{2}}\left(  1+\frac{1}{r_{0}}\right)
^{2}\frac{1}{t_{0}^{\frac{1}{6}}}\right)  }\sum_{r=r_{0}}^{R}1\\
&  =\left(  R-3\right)  \sqrt{5\beta\pi^{\frac{1}{3}}\left(  1+\frac{1}
{t_{0}^{\frac{1}{3}}}+\frac{5c}{4\sqrt{2}\pi^{\frac{1}{6}}}\frac{1}
{t_{0}^{\frac{1}{6}}}\right)  }\\
&  \leq\left(  \frac{1}{\sqrt{2\pi}}t^{\frac{1}{6}}-3\right)  M_{8},
\end{split}
\end{equation}
where%
\[
M_{8}:=\sqrt{5}\pi^{\frac{1}{6}}\sqrt{\left(  \frac{32}{3\sqrt{\pi c}}+\frac
{1}{c}\left(  \frac{6}{\pi}-1\right)  \frac{1}{25\pi^{\frac{2}{3}}}\right)
\left(  1+\frac{1}{t_{0}^{\frac{1}{3}}}+\frac{5c}{4\sqrt{2}\pi^{\frac{1}{6}%
}t_{0}^{\frac{1}{6}}}\right)  }.
\]
As for the $T_{1}$ sum we can operate exactly as in (\ref{T1Estimate}).
Finally, from (\ref{MainSumEstimate}), (\ref{MainSumBrake}), (\ref{T1Estimate}%
), (\ref{T2T3T4T5}), (\ref{T4Estimate}) and (\ref{T5Estimate}), we obtain
\begin{equation}
\left\vert \zeta\left(  \frac{1}{2}+it\right)  \right\vert \leq b_{1}%
t^{\frac{1}{6}}\log t+b_{2}t^{\frac{1}{6}}+b_{3} \label{CorputSum2}%
\end{equation}
with the constants
\begin{align*}
b_{1}  &  =\frac{1}{3\pi^{\frac{1}{6}}}\sqrt{M_{1}+M_{2}},\\
b_{2}  &  =M_{6}+\frac{2M_{5}}{\pi^{\frac{1}{6}}}\sqrt{M_{1}+M_{2}}%
+\frac{M_{8}}{\sqrt{2 \pi}},\\
b_{3}  &  =M_{7}-3M_{8}+1.463t_{0}^{-\frac{1}{4}}+0.127t_{0}^{-\frac{3}{4}}.
\end{align*}
Now, recalling $\phi=1/3$ and setting $c=1.34$ together with $t_{0}%
=8.97\cdot10^{17}$ we finally can finish the proof of Theorem \ref{theorem1}(a)
by
\begin{equation}
\label{FreeEndBound}%
\begin{split}
\left\vert \zeta\left(  1/2+it\right)  \right\vert  &  \leq 0.469028 t^{1/6}\log
t+3.99399t^{1/6}-21.4623\\
&  \leq0.611 t^{1/6}\log t\text{, for }t_{0}\leq t.
\end{split}
\end{equation}

\begin{remark}
We would like to mention that our method also works correctly from the
starting point $t_{0}=7.00\cdot10^{11}$. The computational bound
(\ref{CorputSum2}) then reads
\begin{align*}
\left\vert \zeta\left(  1/2+it\right)  \right\vert  &  \leq 0.470795 t^{1/6}\log
t+4.04972t^{1/6}-21.5437\\
&  \leq0.611 t^{1/6}\log t\text{, for } t_{0}\leq t.
\end{align*}
Thus, there is an overlap inside the last bottleneck region in subsection \ref{LastBottleneck} for the range $7.00\cdot10^{11}
\leq t \leq8.97 \cdot10^{17}$, where both formulas (\ref{CorputSum1}) and
(\ref{CorputSum2}) can be used simultaneously.
\end{remark}

\begin{remark}
We should also mention that the $T_{5}$ sum in (\ref{T5Estimate}) can be
estimated more accurately using a method similar to that presented in
(\ref{SumIntegralEstimate}). However, the overall improvement is marginal.
\end{remark}

\par\noindent Proof for assertion (b) in Theorem \ref{theorem1}:
As for the last region, which starts at $t_{0}=8.97\cdot10^{17}$, we can
improve the last estimate in (\ref{FreeEndBound}) further to
\begin{align*}
\left\vert \zeta\left(  1/2+it\right)  \right\vert  &  \leq 0.469028 t^{1/6}\log
t+3.99399t^{1/6}-21.4623\\
&  \leq0.566 t^{1/6}\log t\text{, for }t_{0}\leq t,
\end{align*}
which gives the mentioned bound in Theorem \ref{theorem1}(b).

\section{Proofs for Lemma \ref{SecondderivativePart1} and Lemma
\ref{SecondderivativePart2}}

\label{Secondderivateproof}

Before we begin, we should mention that our proof of Lemma
\ref{SecondderivativePart1} is similar to the proof of Lemma 2.5 in
\cite{Hiary2}, though it contains some modifications and important extensions.
The proof of Lemma \ref{SecondderivativePart2} follows the general structure
presented in \cite{Hiary2} for Lemma 2.6, combined together with some
refinements and improvements.

\subsection{Proof for Lemma \ref{SecondderivativePart1}}

\noindent(a) Using some standard methods the assertion is easily established.

\noindent(b) We define
\begin{align*}
k  &  =\left\lfloor g^{\prime}\left(  L-1-m\right)  \right\rfloor
-\left\lfloor g^{\prime}\left(  0\right)  \right\rfloor ,\\
C_{0}  &  =\left\lfloor g^{\prime}\left(  0\right)  \right\rfloor ,\\
C_{j}  &  =C_{0}+j\text{ for }1\leq j\leq k+1.
\end{align*}
Then, by monotonicity of $g^{\prime}$, we find that $k$ is a non negative
integer and $C_{k}=\left\lfloor g^{\prime}\left(  L-1-m\right)  \right\rfloor
$. Furthermore, with our definitions for $\varepsilon_{1}$ and $\varepsilon
_{2}$ together with Lemma \ref{Functionestimate}, we obtain
\begin{equation}
\label{Estimatek}%
\begin{split}
k  &  =g^{\prime}\left(  L-1-m\right)  -g^{\prime}\left(  0\right)
+\varepsilon_{1}-\varepsilon_{2}\\
&  \leq\frac{mK}{W}\mu+\varepsilon_{1}-\varepsilon_{2}.
\end{split}
\end{equation}
Now, we select a fixed value $0<\Delta<1/2$. We first assume that $k \geq2$.
The cases $k=0,1$ are treated separately later.

\subsubsection{Proof for the case $k \geq2$}

\noindent We define the following segmentation process for the interval
$\left[  0,L-1-m\right]  $:
\begin{align*}
x_{1}  &  =\max\left\{  \left(  g^{\prime}\right)  ^{-1}\left(  C_{1}%
-\Delta\right)  ,0\right\}  ,\\
x_{j}  &  =\left(  g^{\prime}\right)  ^{-1}\left(  C_{j}-\Delta\right)  \text{
for }2\leq j\leq k+1,\\
y_{j}  &  =\left(  g^{\prime}\right)  ^{-1}\left(  C_{j}+\Delta\right)  \text{
for }0\leq j\leq k-1,\\
y_{k}  &  =\min\left\{  \left(  g^{\prime}\right)  ^{-1}\left(  C_{k}%
+\Delta\right)  ,L-1-m\right\}  .
\end{align*}
One first checks that all these points are well defined. Then, using some
standard arguments together with the observation that $g^{\prime}$ and
$\left(  g^{\prime}\right)  ^{-1}$ are increasing, we arrive at the following
segmentation
\begin{equation}%
\begin{split}
0\leq x_{1}<y_{1}<x_{2}<y_{2}<  &  \cdots<x_{k-1}<y_{k-1}<x_{k}<y_{k}\leq
L-1-m\\
y_{0}<x_{1}  &  \text{ and }y_{k}<x_{k+1}.
\end{split}
\label{Segmentation}%
\end{equation}
The idea behind is that we divide our summation interval into about $2k$
subintervals, where the subintervals will be chosen so that we may apply the
refined Kusmin-Landau bound, Lemma \ref{KusminLandau}, on about half of the
subintervals and the trivial bound for the exponential terms on the remaining
half. Thus we select the subintervals of $\left[  0,L-1-m\right]  $ as
$\left[  0,x_{1}\right)  $, $\left[  x_{1},y_{1}\right]  $, $\left(
y_{1},x_{2}\right)  $,$\ldots$, $\left(  y_{k-1},x_{k-1}\right)  $, $\left[
x_{k},y_{k}\right]  $ and $\left(  y_{k},L-1-m\right]  $ using the trivial
bound for $\left[  x_{j},y_{j}\right]  $ for $1\leq j\leq k$ and the
Kusmin-Landau bound for $\left(  y_{j},x_{j+1}\right)  $ for $1\leq j\leq k-1$
and the subintervals $\left[  0,x_{1}\right)  $ and $\left(  y_{k}%
,L-1-m\right]  $ to be treated separately.

\noindent Next, using again the monotonicity for $\left(  g^{\prime}\right)
^{-1}$, we obtain
\begin{align*}
y_{1}-x_{1}  &  =\left(  g^{\prime}\right)  ^{-1}\left(  C_{1}+\Delta\right)
-\max\left\{  \left(  g^{\prime}\right)  ^{-1}\left(  C_{1}-\Delta\right)
,0\right\} \\
&  =\left(  g^{\prime}\right)  ^{-1}\left(  C_{1}+\Delta\right)  -\left(
g^{\prime}\right)  ^{-1}\left(  \max\left\{  C_{1}-\Delta,g^{\prime}\left(
0\right)  \right\}  \right) \\
&  \leq\left(  \left(  g^{\prime}\right)  ^{-1}\right)  ^{\prime}\left(
\nu_{1}\right)  \left(  C_{1}+\Delta-C_{1}+\Delta\right)
\end{align*}
for some $\max\left\{  C_{1}-\Delta,g^{\prime}\left(  0\right)  \right\}
<\upsilon_{1}<C_{1}+\Delta$. Upon setting $\delta_{1}=\left(  g^{\prime
}\right)  ^{-1}\left(  \upsilon_{1}\right)  $ this implies $x_{1}=\max\left\{
\left(  g^{\prime}\right)  ^{-1}\left(  C_{1}-\Delta\right)  ,0\right\}
<\delta_{1}<\left(  g^{\prime}\right)  ^{-1}\left(  C_{1}+\Delta\right)
=y_{1}$ and together with (\ref{Segmentation}) and Lemma
\ref{SecondderivativePart1}(a) we arrive at
\[
\left(  \left(  g^{\prime}\right)  ^{-1}\right)  ^{\prime}\left(  \nu
_{1}\right)  2\Delta=\frac{2\Delta}{\left(  g^{\prime}\right)  ^{\prime
}\left(  \left(  g^{\prime}\right)  ^{-1}\left(  \upsilon_{1}\right)  \right)
}=\frac{2\Delta}{g^{\prime\prime}\left(  \delta_{1}\right)  }\leq\frac{2W}%
{m}\Delta.
\]
We conclude that $y_{1}-x_{1}\leq2W\Delta/m$. The cases $y_{j}-x_{j}$ for
$2\leq j\leq k-1$ and $y_{k}-x_{k}$ are treated similarly and we obtain for $1
\leq j \leq k$
\[
y_{j}-x_{j} \leq\frac{2W}{m}\Delta.
\]

\noindent Using the triangle inequality we now estimate the sums for the
exponential terms by
\begin{equation}
\left\vert \sum_{x_{j}\leq n\leq y_{j}}e\left(  g\left(  n\right)  \right)
\right\vert \leq\sum_{x_{j}\leq n\leq y_{j}}\left\vert e\left(  g\left(
n\right)  \right)  \right\vert \leq y_{j}-x_{j}+1\leq1+\frac{2W}{m}\Delta.
\label{FirstSumEstimate}%
\end{equation}
\noindent For the remaining subintervals $\left(  y_{j},x_{j+1}\right)  $ for
$1\leq j\leq k-1$ we have by construction $C_{j}+\Delta\leq g^{\prime}\left(
x\right)  \leq C_{j+1}-\Delta=C_{j}+1-\Delta$ and we select in Lemma
\ref{KusminLandau} the values $z=C_{j}$, $U=\Delta$ and $V=1-\Delta$. While we
are here summing over the integers $n$ in the open interval $\left(
y_{j},x_{j+1}\right)  $ and observing that the bound in Lemma
\ref{KusminLandau} is independent of the length of summation we may apply
Lemma \ref{KusminLandau} directly to our case and we obtain
\begin{equation}
\left\vert \sum_{y_{j}<n<x_{j+1}}e\left(  g\left(  n\right)  \right)
\right\vert \leq\frac{1}{\pi}\left(  \frac{1}{\Delta}+\frac{1}{1-\left(
1-\Delta\right)  }\right)  =\frac{2}{\pi\Delta}. \label{SecondSumEstimate}%
\end{equation}
\noindent It remains to consider the boundary intervals $\left[
0,x_{1}\right)  $ and $\left(  y_{k},L-1-m\right]  $.

\noindent We start with $\left[  0,x_{1}\right)  $ for which we derive an
estimate for the sum
\[
\left\vert \sum_{0\leq n<x_{1}}e\left(  g\left(  n\right)  \right)
\right\vert .
\]
We consider the following three cases:

\noindent Case 1: $\varepsilon_{1}\in\left[  0,\Delta\right)  $.%
\begin{align*}
y_{0}  &  =\left(  g^{\prime}\right)  ^{-1}\left(  C_{0}+\Delta\right)
=\left(  g^{\prime}\right)  ^{-1}\left(  \left\lfloor g^{\prime}\left(
0\right)  \right\rfloor +\Delta\right) \\
&  =\left(  g^{\prime}\right)  ^{-1}\left(  g^{\prime}\left(  0\right)
-\varepsilon_{1}+\Delta\right)  >\left(  g^{\prime}\right)  ^{-1}\left(
g^{\prime}\left(  0\right)  \right)  =0.
\end{align*}
Together with (\ref{Segmentation}) we get $0<y_{0}<x_{1}$ and%
\begin{align*}
y_{0}-0  &  =\left(  g^{\prime}\right)  ^{-1}\left(  C_{0}+\Delta\right)  -0\\
&  =\left(  g^{\prime}\right)  ^{-1}\left(  C_{0}+\Delta\right)  -\left(
g^{\prime}\right)  ^{-1}\left(  g^{\prime}\left(  0\right)  \right) \\
&  =\left(  \left(  g^{\prime}\right)  ^{-1}\right)  ^{\prime}\left(  \nu
_{0}\right)  \left(  C_{0}+\Delta-g^{\prime}\left(  0\right)  \right)
\end{align*}
for some $g^{\prime}\left(  0\right)  <\upsilon_{0}<C_{0}+\Delta$. Upon
setting $\delta_{0}=\left(  g^{\prime}\right)  ^{-1}\left(  \upsilon
_{0}\right)  $ this implies $0<\delta_{0}<\left(  g^{\prime}\right)  \left(
C_{0}+\Delta\right)  =y_{0}$ and together with (\ref{Segmentation}) and Lemma
\ref{SecondderivativePart1}(a) we arrive at%
\begin{align*}
y_{0}-0  &  =\frac{1}{\left(  g^{\prime}\right)  ^{\prime}\left(  \left(
g^{\prime}\right)  ^{-1}\left(  \upsilon_{0}\right)  \right)  }\left(
C_{0}+\Delta-g^{\prime}\left(  0\right)  \right) \\
&  =\frac{1}{g^{\prime\prime}\left(  \delta_{0}\right)  }\left(  \left\lfloor
g^{\prime}\left(  0\right)  \right\rfloor +\Delta-g^{\prime}\left(  0\right)
\right) \\
&  =\frac{1}{g^{\prime\prime}\left(  \delta_{0}\right)  }\left(
\Delta-\varepsilon_{1}\right) \\
&  \leq\frac{W}{m}\left(  \Delta-\varepsilon_{1}\right)  .
\end{align*}
Thus we estimate our sum, analogously as in (\ref{SecondSumEstimate}), by
\begin{equation}
\label{LeftBoundaryCase1}%
\begin{split}
\left\vert \sum_{0\leq n<x_{1}}e\left(  g\left(  n\right)  \right)
\right\vert  &  \leq\left\vert \sum_{0\leq n<y_{0}}e\left(  g\left(  n\right)
\right)  \right\vert +\left\vert \sum_{y_{0}\leq n<x_{1}}e\left(  g\left(
n\right)  \right)  \right\vert \\
&  \leq y_{0}-0+1+\frac{2}{\pi\Delta}\\
&  \leq1+\frac{2}{\pi\Delta}+\frac{W}{m}\left(  \Delta-\varepsilon_{1}\right)
.
\end{split}
\end{equation}

\noindent Case 2: $\varepsilon_{1}\in\left[  \Delta,1-\Delta\right]  .$ Here
one simply sees that $y_{0}\leq0$, hence we do not split the sum as before.
Next, for $0\leq n<x_{1}$ we have
\[
C_{0}+\varepsilon_{1}=\left\lfloor g^{\prime}\left(  0\right)  \right\rfloor
+\varepsilon_{1}=g^{\prime}\left(  0\right)  \leq g^{\prime}\left(  n\right)
<g^{\prime}\left(  x_{1}\right)  =C_{1}-\Delta=C_{0}+1-\Delta
\]
and we set in Lemma \ref{KusminLandau} the values $z=C_{0}$, $U=\varepsilon
_{1}$ and $V=1-\Delta$. We obtain%
\begin{equation}
\left\vert \sum_{0\leq n<x_{1}}e\left(  g\left(  n\right)  \right)
\right\vert \leq\frac{1}{\pi}\left(  \frac{1}{\varepsilon_{1}}+\frac{1}%
{\Delta}\right)  . \label{LeftBoundaryCase2}%
\end{equation}

\noindent Case 3: $\varepsilon_{1}\in\left(  1-\Delta,1\right]  .$%
\begin{align*}
\left(  g^{\prime}\right)  ^{-1}\left(  C_{1}-\Delta\right)   &  =\left(
g^{\prime}\right)  ^{-1}\left(  C_{0}+1-\Delta\right) \\
&  =\left(  g^{\prime}\right)  ^{-1}\left(  g^{\prime}\left(  0\right)
+1-\Delta-\varepsilon_{1}\right) \\
&  <\left(  g^{\prime}\right)  ^{-1}\left(  g^{\prime}\left(  0\right)
\right)  =0
\end{align*}
and by definition for $x_{1}$ we obtain $x_{1}=\max\left\{  \left(  g^{\prime
}\right)  ^{-1}\left(  C_{1}-\Delta\right)  ,0\right\}  =0$. So we arrive at
an empty sum%
\begin{equation}
\left\vert \sum_{0\leq n<x_{1}}e\left(  g\left(  n\right)  \right)
\right\vert =0. \label{LeftBoundaryCase3}%
\end{equation}
The right-hand boundary sum%
\[
\left\vert \sum_{y_{k}<n\leq L-1-m}e\left(  g\left(  n\right)  \right)
\right\vert
\]
is estimated in an similar way. We omit the proof but we refer that here the
case $\varepsilon_{2}\in\left[  0,\Delta\right)  $ corresponds to the previous
case $\varepsilon_{1}\in\left(  1-\Delta,1\right]  $, that is%
\begin{equation}
\left\vert \sum_{y_{k}<n\leq L-1-m}e\left(  g\left(  n\right)  \right)
\right\vert =0. \label{RightBoundaryCase1}%
\end{equation}
The case $\varepsilon_{2}\in\left[  \Delta,1-\Delta\right]  $ corresponds to
the previous case $\varepsilon_{1}\in\left[  \Delta,1-\Delta\right]  $, and we
obtain%
\begin{equation}
\left\vert \sum_{y_{k}<n\leq L-1-m}e\left(  g\left(  n\right)  \right)
\right\vert \leq\frac{1}{\pi}\left(  \frac{1}{\Delta}+\frac{1}{1-\varepsilon
_{2}}\right)  . \label{RightBoundaryCase2}%
\end{equation}
Finally, the last remaining case $\varepsilon_{2}\in\left(  1-\Delta,1\right]
$ corresponds to the previous case $\varepsilon_{1}\in\left[  0,\Delta\right)
$, and here we obtain the bound
\begin{equation}
\left\vert \sum_{y_{k}<n\leq L-1-m}e\left(  g\left(  n\right)  \right)
\right\vert \leq1+\frac{2}{\pi\Delta}+\frac{W}{m}\left(  \Delta-\left(
1-\varepsilon_{2}\right)  \right)  . \label{RightBoundaryCase3}%
\end{equation}
Collecting the results (\ref{LeftBoundaryCase1}), (\ref{LeftBoundaryCase2}),
(\ref{LeftBoundaryCase3}), (\ref{RightBoundaryCase1}),
(\ref{RightBoundaryCase2}) and (\ref{RightBoundaryCase3}) together with the
definition for $p\left(  \varepsilon\right)  $ in Lemma
\ref{SecondderivativePart1}(b) we finally arrive for $0\leq\varepsilon
_{1},\varepsilon_{2}\leq1$ at%
\begin{align}
\left\vert \sum_{0\leq n<x_{1}}e\left(  g\left(  n\right)  \right)
\right\vert  &  \leq p\left(  \varepsilon_{1}\right)  ,\label{LeftCase}\\
\left\vert \sum_{y_{k}<n\leq L-1-m}e\left(  g\left(  n\right)  \right)
\right\vert  &  \leq p\left(  1-\varepsilon_{2}\right)  . \label{RightCase}%
\end{align}
Now, we can estimate the main sum. Let $m<L\leq K$. Then, using
(\ref{Estimatek}), (\ref{Segmentation}), (\ref{FirstSumEstimate}),
(\ref{SecondSumEstimate}), (\ref{LeftCase}) and (\ref{RightCase}), we obtain
\begin{equation}%
\begin{split}
&  \left\vert \sum_{n=0}^{L-1-m}e\left(  g\left(  n\right)  \right)
\right\vert \\
&  \leq\left\vert \sum_{0\leq n<x_{1}}e\left(  g\left(  n\right)  \right)
\right\vert +\sum_{j=1}^{k}\left\vert \sum_{x_{j}\leq n\leq y_{j}}e\left(
g\left(  n\right)  \right)  \right\vert \\
&  +\sum_{j=1}^{k-1}\left\vert \sum_{y_{j}<n<x_{j+1}}e\left(  g\left(
n\right)  \right)  \right\vert +\left\vert \sum_{y_{k}<n\leq L-1-m}e\left(
g\left(  n\right)  \right)  \right\vert \\
&  \leq p\left(  \varepsilon_{1}\right)  +k\left(  1+\frac{2W}{m}%
\Delta\right)  +\left(  k-1\right)  \frac{2}{\pi\Delta}+p\left(
1-\varepsilon_{2}\right)  .
\end{split}
\label{Part1SumPart1}%
\end{equation}
Next, inserting (\ref{Estimatek}), we obtain%
\begin{equation}%
\begin{split}
&  \left\vert \sum_{n=0}^{L-1-m}e\left(  g\left(  n\right)  \right)
\right\vert \\
&  \leq k\left(  1+\frac{2W}{m}\Delta+\frac{2}{\pi\Delta}\right)  +p\left(
\varepsilon_{1}\right)  +p\left(  1-\varepsilon_{2}\right)  -\frac{2}%
{\pi\Delta}\\
&  \leq\left(  \frac{mK}{W}\mu+\varepsilon_{1}-\varepsilon_{2}\right)  \left(
1+\frac{2W\Delta}{m}+\frac{2}{\pi\Delta}\right)  +p\left(  \varepsilon
_{1}\right)  +p\left(  1-\varepsilon_{2}\right)  -\frac{2}{\pi\Delta}.
\end{split}
\label{Part1SumPart2}%
\end{equation}
Finally, by applying some simple operations to the last expressions and
inserting the definition of $P\left(  \varepsilon\right)  $ from Lemma
\ref{SecondderivativePart1}(b), we arrive at
\begin{align*}
&  \left\vert \sum_{n=0}^{L-1-m}e\left(  g\left(  n\right)  \right)
\right\vert \\
&  \leq\left(  \frac{mK}{W}\mu-1\right)  \left(  1+\frac{2W\Delta}{m}+\frac
{2}{\pi\Delta}\right)  +P\left(  \varepsilon_{1}\right)  +P\left(
1-\varepsilon_{2}\right)  -\frac{2}{\pi\Delta}.
\end{align*}

\subsubsection{Proof for the case $k =1$}

\noindent Here, we define the following segmentation for the interval $\left[
0,L-1-m\right]  $:
\begin{align*}
x_{1}  &  =\max\left\{  \left(  g^{\prime}\right)  ^{-1}\left(  C_{1}%
-\Delta\right)  ,0\right\}  ,\\
x_{2}  &  =\left(  g^{\prime}\right)  ^{-1}\left(  C_{2}-\Delta\right)  ,\\
y_{0}  &  =\left(  g^{\prime}\right)  ^{-1}\left(  C_{0}+\Delta\right)  ,\\
y_{1}  &  =\min\left\{  \left(  g^{\prime}\right)  ^{-1}\left(  C_{1}%
+\Delta\right)  ,L-1-m\right\}  .
\end{align*}
The setting deviates only in the definition of $y_{1}$ compared to the
previous subsection. Then applying similar standard arguments as before, we
arrive at the following segmentation
\[%
\begin{split}
0  &  \leq x_{1}<y_{1}\leq L-1-m\\
y_{0}  &  <x_{1}\text{ and }y_{1}<x_{2}.
\end{split}
\]
Consequently, we select the subintervals of $\left[  0,L-1-m\right]  $ as
$\left[  0,x_{1}\right)  $, $\left[  x_{1},y_{1}\right]  $ and $\left(
y_{1},L-1-m\right]  $ using the trivial bound for $\left[  x_{1},y_{1}\right]
$ and the subintervals $\left[  0,x_{1}\right)  $ and $\left(  y_{1}%
,L-1-m\right]  $ to be again treated separately. All subintervals are to be
treated analogously as in the previous subsection and we omit the proof
details. We arrive at the same corresponding estimate as in the previous
subsection by%

\begin{align*}
&  \left\vert \sum_{n=0}^{L-1-m}e\left(  g\left(  n\right)  \right)
\right\vert \\
&  \leq\left\vert \sum_{0\leq n<x_{1}}e\left(  g\left(  n\right)  \right)
\right\vert +\left\vert \sum_{x_{1}\leq n\leq y_{1}}e\left(  g\left(
n\right)  \right)  \right\vert +\left\vert \sum_{y_{1}<n\leq L-1-m}e\left(
g\left(  n\right)  \right)  \right\vert \\
&  \leq p\left(  \varepsilon_{1}\right)  +\left(  1+\frac{2W}{m}\Delta\right)
+p\left(  1-\varepsilon_{2}\right) \\
&  \leq p\left(  \varepsilon_{1}\right)  +k\left(  1+\frac{2W}{m}%
\Delta\right)  +\left(  k-1\right)  \frac{2}{\pi\Delta}+p\left(
1-\varepsilon_{2}\right) \\
&  =k\left(  1+\frac{2W}{m}\Delta+\frac{2}{\pi\Delta}\right)  +p\left(
\varepsilon_{1}\right)  +p\left(  1-\varepsilon_{2}\right)  -\frac{2}%
{\pi\Delta}\\
&  \leq\left(  \frac{mK}{W}\mu-1\right)  \left(  1+\frac{2W\Delta}{m}+\frac
{2}{\pi\Delta}\right)  +P\left(  \varepsilon_{1}\right)  +P\left(
1-\varepsilon_{2}\right)  -\frac{2}{\pi\Delta}.
\end{align*}

\subsubsection{Proof for the case $k =0$}

\noindent As it turns out, this case is surprisingly more complicated than
expected. One problem here is the proper definition of the segmentation points where the setting from the previous cases $k=1$ and $k \geq 2$ does not apply. 
Furthermore, in order to obtain a generally valid inequality, we must estimate more precisely here. First we collect%
\begin{align*}
0  &  =k=\left\lfloor g^{\prime}\left(  L-1-m\right)  \right\rfloor
-\left\lfloor g^{\prime}\left(  0\right)  \right\rfloor ,\\
C_{0}  &  =\left\lfloor g^{\prime}\left(  0\right)  \right\rfloor
=\left\lfloor g^{\prime}\left(  L-1-m\right)  \right\rfloor ,\\
C_{1}  &  =C_{0}+1,\\
\varepsilon_{2}  &  =g^{\prime}\left(  L-1-m\right)  -\left\lfloor g^{\prime
}\left(  L-1-m\right)  \right\rfloor \\
&  =g^{\prime}\left(  L-1-m\right)  -\left\lfloor g^{\prime}\left(  0\right)
\right\rfloor \\
&  \geq g^{\prime}\left(  0\right)  -\left\lfloor g^{\prime}\left(  0\right)
\right\rfloor \\
&  =\varepsilon_{1}.
\end{align*}
Therefore, we must address the situation $0\leq\varepsilon_{1}\leq
\varepsilon_{2}\leq1$. We consider the following six cases.

\noindent Case 1: $\varepsilon_{1}\in\left[  0,\Delta\right)  $ and
$\varepsilon_{2}\in\left[  0,\Delta\right)  $.

\noindent We select $x_{1}=0,y_{1}=L-1-m.$ Thus the segmentation of $\left[
0,L-1-m\right]  $ results itself in $\left[  x_{1},y_{1}\right]  .$ Then%
\begin{align*}
y_{1}-x_{1}  &  =L-1-m-0\\
&  =\left(  g^{\prime}\right)  ^{-1}\left(  g^{\prime}\left(  L-1-m\right)
\right)  -\left(  g^{\prime}\right)  ^{-1}\left(  g^{\prime}\left(  0\right)
\right) \\
&  =\left(  \left(  g^{\prime}\right)  ^{-1}\right)  ^{\prime}\left(
\upsilon_{1}\right)  \left(  g^{\prime}\left(  L-1-m\right)  -g^{\prime
}\left(  0\right)  \right)
\end{align*}
for some $g^{\prime}\left(  0\right)  <\upsilon_{1}<g^{\prime}\left(
L-1-m\right)  $. Upon setting $\delta_{1}=\left(  g^{\prime}\right)
^{-1}\left(  \upsilon_{1}\right)  $ this implies $0<\delta_{1}<L-1-m$ and
together with Lemma \ref{SecondderivativePart1}(a) we obtain
\begin{align*}
y_{1}-x_{1}  &  =\left(  \left(  g^{\prime}\right)  ^{-1}\right)  ^{\prime
}\left(  \upsilon_{1}\right)  \left(  \left\lfloor g^{\prime}\left(
L-1-m\right)  \right\rfloor +\varepsilon_{2}-\left\lfloor g^{\prime}\left(
0\right)  \right\rfloor -\varepsilon_{1}\right) \\
&  =\left(  \left(  g^{\prime}\right)  ^{-1}\right)  ^{\prime}\left(
\upsilon_{1}\right)  \left(  \varepsilon_{2}-\varepsilon_{1}\right) \\
&  \leq\left(  \left(  g^{\prime}\right)  ^{-1}\right)  ^{\prime}\left(
\upsilon_{1}\right)  \left(  \Delta-\varepsilon_{1}\right) \\
&  =\frac{\Delta-\varepsilon_{1}}{g^{\prime\prime}\left(  \delta_{1}\right)
}\\
&  \leq\frac{W}{m}\left(  \Delta-\varepsilon_{1}\right)  .
\end{align*}
Estimating your final sum now gives%
\begin{align*}
\left\vert \sum_{n=0}^{L-1-m}e\left(  g\left(  n\right)  \right)  \right\vert
&  \leq\sum_{x_{1}\leq n\leq y_{1}}^{{}}\left\vert e\left(  g\left(  n\right)
\right)  \right\vert =1+y_{1}-x_{1}\\
&  \leq1+\frac{W}{m}\left(  \Delta-\varepsilon_{1}\right)  .
\end{align*}
Recalling the definition for $p\left(  \varepsilon\right)  $ in Lemma
\ref{SecondderivativePart1} together with $1-\varepsilon_{2}\in\left(
1-\Delta,1\right]  $ we obtain%
\begin{align*}
p\left(  \varepsilon_{1}\right)   &  =1+\frac{2}{\pi\Delta}+\frac{W}{m}\left(
\Delta-\varepsilon_{1}\right)  ,\\
p\left(  1-\varepsilon_{2}\right)   &  =0.
\end{align*}
From this we finally get%
\begin{align*}
\left\vert \sum_{n=0}^{L-1-m}e\left(  g\left(  n\right)  \right)  \right\vert
&  \leq1+\frac{W}{m}\left(  \Delta-\varepsilon_{1}\right) \\
&  =p\left(  \varepsilon_{1}\right)  +p\left(  1-\varepsilon_{2}\right)
-\frac{2}{\pi\Delta}.
\end{align*}

\noindent Case 2: $\varepsilon_{1}\in\left[  0,\Delta\right)  $ and
$\varepsilon_{2}\in\left[  \Delta,1-\Delta\right]  $.

\noindent We select $y_{0}=\left(  g^{^{\prime}}\right)  ^{-1}\left(
C_{0}+\Delta\right)  $, $x_{1}=L-1-m$. One easily checks that $0<y_{0}\leq
x_{1}=L-1-m$. Thus the segmentation of $\left[  0,L-1-m\right]  $ results
itself in $\left[  0,y_{0}\right]  ,\left(  y_{0},x_{1}\right]  $ using the
trivial bound for $\left[  0,y_{0}\right]  $ and the Kusmin-Landau bound for
$\left(  y_{0},x_{1}\right]  $. Following similar proof steps as before and
using Lemma \ref{KusminLandau} together with values $z=C_{0}$, $U=\Delta$ and
$V=\varepsilon_{2}$ we estimate your final sum by%
\begin{align*}
\left\vert \sum_{n=0}^{L-1-m}e\left(  g\left(  n\right)  \right)  \right\vert
&  \leq\left\vert \sum_{0\leq n\leq y_{0}}^{{}}e\left(  g\left(  n\right)
\right)  \right\vert +\left\vert \sum_{y_{0}<n\leq x_{1}}^{{}}e\left(
g\left(  n\right)  \right)  \right\vert \\
&  \leq1+\frac{W}{m}\left(  \Delta-\varepsilon_{1}\right)  +\frac{1}{\pi
}\left(  \frac{1}{\Delta}+\frac{1}{1-\varepsilon_{2}}\right)  .
\end{align*}
Together with $1-\varepsilon_{2}\in\left[  \Delta,1-\Delta\right]  $ we obtain%
\begin{align*}
p\left(  \varepsilon_{1}\right)   &  =1+\frac{2}{\pi\Delta}+\frac{W}{m}\left(
\Delta-\varepsilon_{1}\right)  ,\\
p\left(  1-\varepsilon_{2}\right)   &  =\frac{1}{\pi}\left(  \frac{1}{\Delta
}+\frac{1}{1-\varepsilon_{2}}\right)  .
\end{align*}
From this we finally get%
\begin{align*}
\left\vert \sum_{n=0}^{L-1-m}e\left(  g\left(  n\right)  \right)  \right\vert
&  \leq1+\frac{W}{m}\left(  \Delta-\varepsilon_{1}\right)  +\frac{1}{\pi
}\left(  \frac{1}{\Delta}+\frac{1}{1-\varepsilon_{2}}\right) \\
&  =p\left(  \varepsilon_{1}\right)  +p\left(  1-\varepsilon_{2}\right)
-\frac{2}{\pi\Delta}.
\end{align*}

\noindent Case 3: $\varepsilon_{1}\in\left[  0,\Delta\right)  $ and
$\varepsilon_{2}\in\left(  1-\Delta,1\right]  $.

\noindent We select $y_{0}=\left(  g^{^{\prime}}\right)  ^{-1}\left(
C_{0}+\Delta\right)  $, $x_{1}=\left(  g^{^{\prime}}\right)  ^{-1}\left(
C_{1}-\Delta\right)  $ and $y_{1}=L-1-m$. Again one easily checks that
$0<y_{0}<x_{1}<y_{1}=L-1-m$. Thus the segmentation of $\left[  0,L-1-m\right]
$ results itself in $\left[  0,y_{0}\right]  ,\left(  y_{0},x_{1}\right)  $
and $\left[  x_{1},y_{1}\right]  $ using the trivial bound for $\left[
0,y_{0}\right]  $, $\left[  x_{1},y_{1}\right]  $ and the Kusmin-Landau bound
for $\left(  y_{0},x_{1}\right)  $. The final estimate then follows the lines
before and we arrive at%
\begin{align*}
\left\vert \sum_{n=0}^{L-1-m}e\left(  g\left(  n\right)  \right)  \right\vert
&  \leq\left\vert \sum_{0\leq n\leq y_{0}}^{{}}e\left(  g\left(  n\right)
\right)  \right\vert +\left\vert \sum_{y_{0}<n<x_{1}}^{{}}e\left(  g\left(
n\right)  \right)  \right\vert +\left\vert \sum_{x_{1}\leq n\leq y_{1}}^{{}%
}e\left(  g\left(  n\right)  \right)  \right\vert \\
&  \leq1+\frac{W}{m}\left(  \Delta-\varepsilon_{1}\right)  +\frac{2}{\pi
\Delta}+1+\frac{W}{m}\left(  \Delta-\left(  1-\varepsilon_{2}\right)  \right)
.
\end{align*}
Together with $1-\varepsilon_{2}\in\left[  0,\Delta\right)  $ we obtain%
\begin{align*}
p\left(  \varepsilon_{1}\right)   &  =1+\frac{2}{\pi\Delta}+\frac{W}{m}\left(
\Delta-\varepsilon_{1}\right)  ,\\
p\left(  1-\varepsilon_{2}\right)   &  =1+\frac{2}{\pi\Delta}+\frac{W}%
{m}\left(  \Delta-\left(  1-\varepsilon_{2}\right)  \right)  .
\end{align*}
From this we finally get%
\begin{align*}
\left\vert \sum_{n=0}^{L-1-m}e\left(  g\left(  n\right)  \right)  \right\vert
&  \leq1+\frac{W}{m}\left(  \Delta-\varepsilon_{1}\right)  +\frac{2}{\pi
\Delta}+1+\frac{W}{m}\left(  \Delta-\left(  1-\varepsilon_{2}\right)  \right)
\\
&  =p\left(  \varepsilon_{1}\right)  +p\left(  1-\varepsilon_{2}\right)
-\frac{2}{\pi\Delta}.
\end{align*}

\noindent Case 4: $\varepsilon_{1}\in\left[  \Delta,1-\Delta\right]  $ and
$\varepsilon_{2}\in\left[  \Delta,1-\Delta\right]  $.

\noindent We select $y_{0}=0$, $x_{1}=L-1-m$. The segmentation of $\left[
0,L-1-m\right]  $ results itself in $\left[  y_{0},x_{1}\right]  $. Here we
use Lemma \ref{KusminLandau} together with values $z=C_{0}$, $U=\varepsilon
_{1}$ and $V=\varepsilon_{2}$. We obtain%
\[
\left\vert \sum_{n=0}^{L-1-m}e\left(  g\left(  n\right)  \right)  \right\vert
\leq\frac{1}{\pi}\left(  \frac{1}{\varepsilon_{1}}+\frac{1}{1-\varepsilon_{2}%
}\right)  .
\]
Together with $1-\varepsilon_{2}\in\left[  \Delta,1-\Delta\right]  $ we obtain%
\begin{align*}
p\left(  \varepsilon_{1}\right)   &  =\frac{1}{\pi}\left(  \frac{1}{\Delta
}+\frac{1}{\varepsilon_{1}}\right) \\
p\left(  1-\varepsilon_{2}\right)   &  =\frac{1}{\pi}\left(  \frac{1}{\Delta
}+\frac{1}{1-\varepsilon_{2}}\right)  .
\end{align*}
From this we finally get%
\begin{align*}
\left\vert \sum_{n=0}^{L-1-m}e\left(  g\left(  n\right)  \right)  \right\vert
&  \leq\frac{1}{\pi}\left(  \frac{1}{\varepsilon_{1}}+\frac{1}{1-\varepsilon
_{2}}\right) \\
&  =p\left(  \varepsilon_{1}\right)  +p\left(  1-\varepsilon_{2}\right)
-\frac{2}{\pi\Delta}.
\end{align*}

\noindent Case 5: $\varepsilon_{1}\in\left[  \Delta,1-\Delta\right]  $ and
$\varepsilon_{2}\in\left(  1-\Delta,1\right]  $.

\noindent We select $y_{0}=0$, $x_{1}=\left(  g^{^{\prime}}\right)
^{-1}\left(  C_{1}-\Delta\right)  $ and $y_{1}=L-1-m$. One checks that
$0=y_{0}\leq x_{1}\leq y_{1}=L-1-m$. Thus the segmentation of $\left[
0,L-1-m\right]  $ results itself in $\left[  y_{0},x_{1}\right)  $ and
$\left[  x_{1},y_{1}\right]  $. Again, we use the Kusmin-Landau bound for
$\left[  y_{0},x_{1}\right)  $ and the trivial bound for the remaining
interval $\left[  x_{1},y_{1}\right]  $. Analogue computations as before then
lead us to%
\begin{align*}
\left\vert \sum_{n=0}^{L-1-m}e\left(  g\left(  n\right)  \right)  \right\vert
&  \leq\left\vert \sum_{y_{0}\leq n<x_{1}}^{{}}e\left(  g\left(  n\right)
\right)  \right\vert +\left\vert \sum_{x_{1}\leq n\leq y_{1}}^{{}}e\left(
g\left(  n\right)  \right)  \right\vert \\
&  \leq\frac{1}{\pi}\left(  \frac{1}{\Delta}+\frac{1}{\varepsilon_{1}}\right)
+1+\frac{W}{m}\left(  \Delta-\left(  1-\varepsilon_{2}\right)  \right)  .
\end{align*}
Together with $1-\varepsilon_{2}\in\left[  0,\Delta\right)  $ we obtain%
\begin{align*}
p\left(  \varepsilon_{1}\right)   &  =\frac{1}{\pi}\left(  \frac{1}{\Delta
}+\frac{1}{\varepsilon_{1}}\right) \\
p\left(  1-\varepsilon_{2}\right)   &  =1+\frac{2}{\pi\Delta}+\frac{W}%
{m}\left(  \Delta-\left(  1-\varepsilon_{2}\right)  \right)  .
\end{align*}
From this we finally get%
\begin{align*}
\left\vert \sum_{n=0}^{L-1-m}e\left(  g\left(  n\right)  \right)  \right\vert
&  \leq\frac{1}{\pi}\left(  \frac{1}{\Delta}+\frac{1}{\varepsilon_{1}}\right)
+1+\frac{W}{m}\left(  \varepsilon_{2}-\left(  1-\Delta\right)  \right) \\
&  =p\left(  \varepsilon_{1}\right)  +p\left(  1-\varepsilon_{2}\right)
-\frac{2}{\pi\Delta}.
\end{align*}

\noindent Case 6: $\varepsilon_{1}\in\left(  1-\Delta,1\right]  $ and
$\varepsilon_{2}\in\left(  1-\Delta,1\right]  $.

\noindent This case runs analogue to the first case. We select $x_{1}=0$,
$y_{1}=L-1-m$. The final procedure then results in%
\begin{align*}
\left\vert \sum_{n=0}^{L-1-m}e\left(  g\left(  n\right)  \right)  \right\vert
&  \leq\sum_{x_{1}\leq n\leq y_{1}}^{{}}\left\vert e\left(  g\left(  n\right)
\right)  \right\vert =1+y_{1}-x_{1}\\
&  \leq1+\frac{W}{m}\left(  \Delta-\left(  1-\varepsilon_{2}\right)  \right)
.
\end{align*}
Together with $1-\varepsilon_{2}\in\left[  0,\Delta\right)  $ we obtain%
\begin{align*}
p\left(  \varepsilon_{1}\right)   &  =0\\
p\left(  1-\varepsilon_{2}\right)   &  =1+\frac{2}{\pi\Delta}+\frac{W}%
{m}\left(  \Delta-\left(  1-\varepsilon_{2}\right)  \right)  .
\end{align*}
From this we finally get%
\begin{align*}
\left\vert \sum_{n=0}^{L-1-m}e\left(  g\left(  n\right)  \right)  \right\vert
&  \leq1+\frac{W}{m}\left(  \Delta-\left(  1-\varepsilon_{2}\right)  \right)
\\
&  =p\left(  \varepsilon_{1}\right)  +p\left(  1-\varepsilon_{2}\right)
-\frac{2}{\pi\Delta}.
\end{align*}

\noindent Now, collecting all intermediate results, we finally arrive in all
six cases at the same corresponding estimate as in the previous subsections by%

\begin{align*}
&  \left\vert \sum_{n=0}^{L-1-m}e\left(  g\left(  n\right)  \right)
\right\vert \\
&  \leq p\left(  \varepsilon_{1}\right)  +p\left(  1-\varepsilon_{2}\right)
-\frac{2}{\pi\Delta}\\
&  =k\left(  1+\frac{2W}{m}\Delta+\frac{2}{\pi\Delta}\right)  +p\left(
\varepsilon_{1}\right)  +p\left(  1-\varepsilon_{2}\right)  -\frac{2}%
{\pi\Delta}\\
&  \leq\left(  \frac{mK}{W}\mu-1\right)  \left(  1+\frac{2W\Delta}{m}+\frac
{2}{\pi\Delta}\right)  +P\left(  \varepsilon_{1}\right)  +P\left(
1-\varepsilon_{2}\right)  -\frac{2}{\pi\Delta},
\end{align*}
which in turn represents the final result in Lemma \ref{SecondderivativePart1}%
(b). Finally, we summarize the same bound in all cases $k=0,1$ and $k\geq2$
and the main result in Lemma \ref{SecondderivativePart1}(b) is established.

\subsection{Proof for Lemma \ref{SecondderivativePart2}}

\noindent We split the proof in five different parts.

\noindent(a) First, we show that $R\geq4$. For our values $t\geq
t_{0}=5\cdot10^{6}$ together with $\phi\in\left[  1/3, 0.35\right]  $ some
standard calculations show that by an monotonicity argument we obtain%
\begin{align*}
R  &  =\left\lfloor \frac{n_{1}}{K}\right\rfloor >\frac{\frac{1}{\sqrt{2\pi}%
}\sqrt{t}-1}{t^{\phi}+1}-1\\
&  \geq\frac{\frac{1}{\sqrt{2\pi}}\sqrt{t_{0}}-1}{t_{0}^{\phi}+1}%
-1=\frac{\sqrt{\frac{t_{0}}{\sqrt{2\pi}}}-1}{t_{0}^{\phi}+1}-1.
\end{align*}
By definition $R$ is a positive integer and thus%
\[
R\geq\left\lceil \frac{\sqrt{\frac{t_{0}}{\sqrt{2\pi}}}-1}{t_{0}^{\phi}%
+1}-1\right\rceil =:R_{0}.
\]
Now, a direct calculation leads to
\[
R_{0}\geq\left\lceil \frac{\sqrt{\frac{t_{0}}{\sqrt{2\pi}}}-1}{t_{0}^{0.35}%
+1}-1\right\rceil =\lceil 3.0115\ldots \rceil = 4.
\]
\noindent(b) Next, we show that $\min\left\{  M,K-1\right\}  =M$.

\noindent Applying some standard calculations we easily arrive at
\[
3\pi^{\frac{1}{3}}\left(  \frac{1}{\sqrt{2\pi}}t^{\frac{1}{6}}+1\right)
+2\leq t^{\frac{1}{3}}
\]
which holds in our range for $t$. In fact, the inequality holds for all $t
\geq2012$. Next, using $\phi\geq1/3$, we estimate%
\[
R=\left\lfloor \frac{n_{1}}{K}\right\rfloor \leq\frac{1}{\sqrt{2\pi}}%
t^{\frac{1}{2}-\phi}\leq\frac{1}{\sqrt{2\pi}}t^{\frac{1}{6}}.
\]
Using $0<c<3$ and $r\leq R$ we put together and obtain%
\begin{align*}
1  &  \geq\frac{1}{t^{\frac{1}{3}}}\left(  3\pi^{\frac{1}{3}}\left(  \frac
{1}{\sqrt{2\pi}}t^{\frac{1}{6}}+1\right)  +2\right) \\
&  \geq\frac{1}{t^{\frac{1}{3}}}\left(  c\pi^{\frac{1}{3}}\left(  R+1\right)
+2\right) \\
&  \geq c\frac{\pi^{\frac{1}{3}}\left(  r+1\right)  }{t^{\frac{1}{3}}}%
+\frac{2}{\left\lceil t^{\phi}\right\rceil }\\
&  =\frac{1}{K}\left(  cW^{\frac{1}{3}}+2\right)  .
\end{align*}
From this we obtain $K\geq cW^{1/3}+2>\left\lceil cW^{1/3}\right\rceil +1=M+1$
which shows that $K-1 \geq M$.

\noindent(c) Next, we select $\Delta=\sqrt{m/\pi W}$ and we show that
$0<\Delta< 1/6$. Together with $4\leq r$ and $\phi\geq1/3$ we have%
\[
\frac{1}{W}=\frac{t}{\pi\left(  r+1\right)  ^{3}K^{3}}\leq\frac{t}{125\pi
}\frac{1}{\left\lceil t^{\phi}\right\rceil ^{3}}\leq\frac{1}{125\pi}\frac
{1}{t^{3\phi-1}}\leq\frac{1}{125\pi}.
\]
Now, using $m\leq M,1/W \leq1/125\pi$ and $0<c<3$, we estimate the value
$\Delta$ by
\begin{align*}
\Delta &  =\sqrt{\frac{m}{\pi W}}\leq\sqrt{\frac{\left\lceil cW^{\frac{1}{3}%
}\right\rceil }{\pi W}}<\sqrt{\frac{cW^{\frac{1}{3}}+1}{\pi W}}\\
&  \leq\pi^{-\frac{1}{2}}\pi^{-\frac{1}{3}}\frac{1}{5}\sqrt{c+1}<\frac{2}%
{5\pi\frac{5}{6}}=0.154088<\frac{1}{6}.
\end{align*}
(d) Now, by considering three cases, will show that%
\[
P\left(  \varepsilon\right)  \leq1+\frac{5}{\pi\Delta}-\frac{3}{\pi
},\varepsilon\in\left[  0,1\right]  .
\]
Case 1: $\varepsilon\in\left[  0,\Delta\right)  $. Recalling the definition
for $P\left(  \varepsilon\right)  $ together with $\Delta W/m = 1/\pi\Delta$
for our choice of $\Delta$, we find that%
\begin{align*}
P\left(  \varepsilon\right)   &  =\left(  1+\frac{4}{\pi\Delta}\right)
\varepsilon+1+\frac{2}{\pi\Delta}+\frac{1}{\pi\Delta}-\frac{\varepsilon}%
{\pi\Delta^{2}}\\
&  =1+\frac{3}{\pi\Delta}+\frac{\varepsilon}{\pi}\left(  -\frac{1}{\Delta^{2}%
}+\frac{4}{\Delta}+\pi\right)  .
\end{align*}
Then the function $f\left(  x\right)  =-1/x^{2}+4/x+\pi$ is increasing in
$\left(  0, 1/6\right)  $ and thus we get%
\begin{align*}
P\left(  \varepsilon\right)   &  \leq1+\frac{3}{\pi\Delta}+\frac{\varepsilon
}{\pi}f\left(  \frac{1}{6}\right) \\
&  =1+\frac{3}{\pi\Delta}+\frac{\varepsilon}{\pi}\left(  \pi-12\right) \\
&  \leq1+\frac{3}{\pi\Delta}.
\end{align*}
We collect%
\begin{equation}
P\left(  \varepsilon\right)  \leq1+\frac{3}{\pi\Delta},\varepsilon\in\left[
0,\Delta\right)  . \label{P_Case1}%
\end{equation}
Case 2: $\varepsilon\in\left[  \Delta,1-\Delta\right]  .$ Here we have
\[
P\left(  \varepsilon\right)  =\frac{1}{\pi\Delta}+\left(  1+\frac{4}{\pi
\Delta}\right)  \varepsilon+\frac{1}{\pi\varepsilon}.
\]
Considering the function $f\left(  x\right)  =\left(  1+4/\pi\Delta\right)
x+1/\pi x$ one easily checks that $f$ is convex, hence%
\[
P\left(  \varepsilon\right)  =\frac{1}{\pi\Delta}+f\left(  \varepsilon\right)
\leq\frac{1}{\pi\Delta}+\max\left\{  f\left(  \Delta\right)  ,f\left(
1-\Delta\right)  \right\}  .
\]
Some standard operations now show that $f\left(  1-\Delta\right)  >f\left(
\Delta\right)  $, and we estimate
\[
P\left(  \varepsilon\right)  \leq1+\frac{5}{\pi\Delta}-\Delta-\frac{4}{\pi
}+\frac{1}{\pi\left(  1-\Delta\right)  }.
\]
Similarly as before we consider the function $f\left(  x\right)  =1/\pi\left(
1-x\right)  -x$ and we again find that $f$ is convex in $\left(  0,
1/6\right)  $ thus%
\begin{align*}
f\left(  x\right)   &  \leq\max\left\{  f\left(  0\right)  ,f\left(  \frac
{1}{6}\right)  \right\}  =\max\left\{  \frac{1}{\pi},\frac{1}{\pi\left(
1-\frac{1}{6}\right)  }-\frac{1}{6}\right\} \\
&  =\max\left\{  \frac{1}{\pi},\frac{6}{5\pi}-\frac{1}{6}\right\}  =\frac
{1}{\pi}.
\end{align*}
Thus we find $P\left(  \varepsilon\right)  \leq1+5/\pi\Delta-4/\pi+f\left(
\Delta\right)  \leq1+5/\pi\Delta-4/\pi+1/\pi$. We collect%
\begin{equation}
P\left(  \varepsilon\right)  \leq1+\frac{5}{\pi\Delta}-\frac{3}{\pi
},\varepsilon\in\left[  \Delta,1-\Delta\right]  . \label{P_Case2}%
\end{equation}
Case 3: $\varepsilon\in\left(  1-\Delta,1\right]  .$ In this situation we
simply find the estimate%
\begin{equation}
P\left(  \varepsilon\right)  \leq1+\frac{4}{\pi\Delta},\varepsilon\in\left(
1-\Delta,1\right]  . \label{P_Case3}%
\end{equation}
Putting together (\ref{P_Case1}), (\ref{P_Case2}) and (\ref{P_Case3}) we
arrive at our final estimate
\begin{equation}
P\left(  \varepsilon\right)  \leq1+\frac{5}{\pi\Delta}-\frac{3}{\pi
},\varepsilon\in\left[  0,1\right]  . \label{P_Estimate}%
\end{equation}
(e) Finally, let $L\leq K,1\leq m\leq\min\left(  M,K-1\right)  =M$ together
with our $\Delta=\sqrt{m/\pi W}$.

\noindent Case 1: $1\leq m<L$. Here, we can apply Lemma
\ref{SecondderivativePart1}(b) together with (\ref{P_Estimate}). Recall that
for our choice for $\Delta$ we have $\Delta W/m = 1/\pi\Delta$. Thus we
estimate%
\begin{align*}
\left\vert \sum_{n=0}^{L-1-m}e\left(  g\left(  n\right)  \right)  \right\vert
&  \leq\left(  \frac{mK}{W}\mu-1\right)  \left(  1+\frac{4}{\pi\Delta}\right)
+2\left(  1+\frac{5}{\pi\Delta}-\frac{3}{\pi}\right)-\frac{2}{\pi\Delta} \\
&  =\frac{mK}{W}\mu\left(  1+\frac{4}{\pi\Delta}\right)  +2+\frac{10}%
{\pi\Delta}-\frac{6}{\pi}-1-\frac{4}{\pi\Delta}-\frac{2}{\pi\Delta}\\
&  =\frac{mK}{W}\mu\frac{4\Delta W}{m}+\frac{mK}{W}\mu+\frac{4\Delta W}%
{m}+1-\frac{6}{\pi}\\
&  =\frac{4\mu K}{\sqrt{\pi W}}m^{\frac{1}{2}}+\frac{\mu K}{W}m+4\sqrt
{\frac{W}{\pi}}m^{-\frac{1}{2}}+1-\frac{6}{\pi}.
\end{align*}
Case 2: $m\geq L$. In this case the left-hand sum is an empty sum with value
$0$. Together with $0<\Delta< 1/6$ this sum is correctly estimated by the
right-hand side, since we have%
\begin{align*}
&  \frac{4\mu K}{\sqrt{\pi W}}m^{\frac{1}{2}}+\frac{\mu K}{W}m+4\sqrt{\frac
{W}{\pi}}m^{-\frac{1}{2}}+1-\frac{6}{\pi}\\
&  \geq0+0+4\sqrt{\frac{W}{\pi}}m^{-\frac{1}{2}}+1-\frac{6}{\pi}\\
&  =\frac{4}{\pi}\sqrt{\frac{W\pi}{m}}+1-\frac{6}{\pi}\\
&  =\frac{4}{\pi}\frac{1}{\Delta}+1-\frac{6}{\pi} > \frac{24}{\pi} + 1-\frac{6}{\pi} > 0.
\end{align*}

\section{Concluding remarks}

We believe that your analysis in the proof steps for Theorem \ref{theorem1} is
more or less optimized. Some further improvements on the constant $0.611$ lead
us to the limitations of the Riemann-Siegel bound together with the transition
into the intermediate bottleneck region. We believe that we have reached the
front end, where there is little additional space for further improvements. 
This follows from the observation that progress for the individual intermediate bottleneck regions is disproportionately small when the $t_{0}$ values are tightened. Therefore, hundreds of bottleneck regions would have to be treated numerically. That would involve a tremendous amount of additional computational work.

\medskip\noindent We will discuss three theoretical approaches that may or may not result in marginal or modest improvements.

\medskip\noindent First, the bound $0 < \Delta< 1/6$ can be slightly improved.
This follows from the proof for Lemma \ref{SecondderivativePart2}, part (c).
Here, we can improve to $\Delta< 2/(5 \pi^{5/6})=0.154088\ldots$. 
Unfortunately, this does not lead to an improvement in the estimate (\ref{P_Estimate}) and consequently does not lead to an improvement for the constant term 
in Lemma \ref{SecondderivativePart2}.

\noindent Second, our estimate for the $T_{3}$ sum in subsection
\ref{Bottleneckregion} may be inefficient. This last piece contains
$n_{1}-KR+1$ terms, yet we have bounded it as though it contains $K$ terms. In
situations where $K$ is much bigger than $n_{1}-KR+1$ this approach appears
wasteful. However, it is unclear how this could be managed in a better way.

\noindent Third, the summation process for the $T_{2}$ and $T_{3}$ sums in
(\ref{MainSumBrake}), (\ref{T2Estimate}) and (\ref{T2T3Estimate}) via partial
summation together with Abel's inequality is probably inefficient to some
extent. Here the denominators for the individual segments are always estimated
by their smallest segment value. Though this process leads directly to an
explicit tractable exponential sum, it is unclear whether standard methods can
improve it.

\medskip\noindent The overall fundamental problem is that the exponential sums are estimated inefficiently, 
meaning that too many cancellations are overlooked. However, we are unaware of a mechanism that would enable us to exploit these cancellation terms more effectively.

\medskip\noindent Michael Revers \newline Department of Mathematics \newline
University Salzburg\newline Hellbrunnerstrasse 34 \newline5020 Salzburg,
AUSTRIA \newline e-mail: michael.revers@plus.ac.at \newline16-digit ORCID
identifier: 0000-0002-9668-5133


\begin{thebibliography}{99}                                                                                               %


\bibitem {Bourgain}J. Bourgain, Decoupling, exponential sums and the Riemann
zeta function, J. Am. Math. Soc. 30 (1) (2016) 205-224.

\bibitem {ChengGraham}Y.F. Cheng, S.W. Graham, Explicit estimates for the
Riemann zeta function, Rocky Mountain L. Math 34 (4) (2004) 1261-1280.

\bibitem {Ford}K. Ford, \textit{Zero-free regions for the Riemann zeta
function}, in: Number theory for the Millenium, II, Urbana, IL, 2000, A K
Peters, Natick, MA, 2002, pp. 25-56

\bibitem {Gabke}W. Gabke, Neue Herleitung und explizite Restabschätzung der
Riemann Siegel Formel (Ph.D. thesis), Göttingen, 1979.

\bibitem {GrahamKolesnik}S.W. Graham, G. Kolesnik, \textit{van der Corput's
Method of Exponential Sums}, in: London Mathematical Society Lecture Note
Series, Vol. 126, Cambridge University Press, Cambridge 1991.

\bibitem {Hasanalizade}E. Hasanalizade, Q. Shen, P.J. Wong, Counting zeros of
the Riemann zeta function, J. Number Theory 235 (2021) 219-241.

\bibitem {Hiary1}G.A. Hiary, An explicit van der Corput estimate for
$\xi\left(  1/2+it\right)  $, Indag. Math. 27 (2) (2016) 524-533.

\bibitem {Hiary2}G.A. Hiary, D. Patel, A. Yang, An improved esplicit estimate
for $\xi\left(  1/2+it\right)$, J. Number Theory 256 (2024) 195-217.

\bibitem {Huxley}M.N. Huxley, Exponential sums and the Riemann zeta function,
V, Proc. Lond. Math. Soc. (3) 90 (1) (2005) 1-41.

\bibitem {Kadiri}H. Kadiri, A. Lumley, N. Ng, Explicit zero density for the
Riemann zeta function, J. Math. Anal. Appl. 465 (1) (2018) 22-46.

\bibitem {Kusmin}R. Kusmin, Sur quelques inégalités trigonométriques, Soc.
Phys.- Math. Léeningr. 1 (1927) 233-239.

\bibitem {Landau}E. Landau, Über eine trigonometrische Summe, Nachr. Ges.
Wiss. Gött. (1928) 21-24.

\bibitem {Lehman}R.S. Lehman, On the distribution of zeros of the Riemann
zeta-function, Proc. Lond. Math. Soc. S3-20 (2) (1970) 303-320.

\bibitem {Lindelof}E. Lindelöf, Quelques remarques sur la croissance de la
fonction, Bull. Sci. Math., 32 (1908) 341--356.

\bibitem {Mossinghoff}M.J. Mossinghoff, T.S. Trudgian, A. Yang, Explicit
zero-free regions for the Riemann zeta-function, Res. Number Theory 10 (1)
(2024) xxxx

\bibitem {Patel}D. Patel, An explicit upper bound for $\left\vert \xi\left(
1+it\right)  \right\vert $, Indag. Math. 33 (5) (2022) 1012-1032.

\bibitem {PlattTrudgian}D.J. Platt, T.S. Trudgian, An improved explicit bound
on $\left\vert \xi\left(  1/2+it\right)  \right\vert $, J. Number Theory 147
(2015) 842-851.

\bibitem {Titchmarch}Edited and with a preface by E.C. Titchmarch, in: D.R.
Heath-Brown (Ed.), \textit{The theory of the Riemann zeta-function}, second
ed., The Clarendon Press Oxford University Press, New York, 1986.

\bibitem {Trudgian}T.S. Trudgian An improved upper bound for the argument of
the Riemann zeta-function on the critical line II, J. Number theory 134 (2014) 280-292.

\bibitem {Corput1}J.G. van der Corput, Neue zahlentheoretische Abschätzungen,
Math. Z. 29 (1929) 397-426.

\bibitem {Corput2}J.G. van der Corput, J.F. Koksma, Sur l'odre de grandeur de
la fonction $\xi\left(  s\right)  $ de Riemann dans la brande critique, Ann.
Toulouse (3) 22, 1930, 1-39.

\bibitem {Yang}A. Yang, Explicit bounds on $\xi\left(  s \right)  $ in the
critical strip and a zero-free region, J. Math. Anal. Appl. 534 (2) (2024)
128124.
\end{thebibliography}
\end{document}